\documentclass[11pt]{article}
\usepackage{amsmath,amsthm}
\usepackage{amssymb,mathrsfs}
\usepackage{bm,mathtools}
\usepackage{graphics}
\usepackage{tikz}
\usepackage{float}
\usepackage{enumitem}
\usepackage{array}
\usetikzlibrary{calc}
\usepackage{xcolor}

\usepackage[normalem]{ulem}

\usepackage{geometry}
\usepackage[colorlinks=true,
linkcolor=blue,citecolor=blue,
urlcolor=blue]{hyperref}

\makeatletter
\def\@seccntDot{.}
\def\@seccntformat#1{\csname the#1\endcsname\@seccntDot\hskip 0.5em}
\renewcommand\section{\@startsection{section}{1}{\z@}%
	{18\p@ \@plus 6\p@ \@minus 3\p@}%
	{9\p@ \@plus 6\p@ \@minus 3\p@}%
	{\LARGE\bfseries\boldmath}}
\renewcommand\subsection{\@startsection{subsection}{2}{\z@}%
	{12\p@ \@plus 6\p@ \@minus 3\p@}%
	{3\p@ \@plus 6\p@ \@minus 3\p@}%
	{\bfseries\boldmath}}
\renewcommand\subsubsection{\@startsection{subsubsection}{3}{\z@}%
	{12\p@ \@plus 6\p@ \@minus 3\p@}%
	{\p@}%
	{\bfseries\boldmath}}
\makeatother

\renewenvironment{proof}{\noindent\textbf{Proof.}}{\hfill $\blacksquare$\par}

\usepackage{microtype}
\usepackage{float}

\newcommand{\keywords}[1]{%
	\par\medskip
	\noindent{\small\textbf{Keywords:} #1\par}
	\medskip
}

\theoremstyle{plain}
\newtheorem{theorem}{Theorem}[section]
\newtheorem{lemma}[theorem]{Lemma}

\newtheorem{proposition}[theorem]{Proposition}
\newtheorem{problem}[theorem]{Problem}

\newtheorem{Orientation}{Orientation}
\newtheorem{observation}{Observation}
\newtheorem{claim}{Claim}

\theoremstyle{definition}
\newtheorem{definition}[theorem]{Definition}

\numberwithin{equation}{section}
\allowdisplaybreaks
\title{\bf An improved upper bound on the oriented diameter of graphs with diameter $4$}

\author{
	Jifu Lin\thanks{Department of Mathematics, East China Normal University, Shanghai 200241, China, e-mail: {\tt jifulin01@163.com}.}
	\and 
	Xiaolin Wang\thanks{School of Mathematics and Statistics, Fuzhou University, Fuzhou 350108, China, e-mail: {\tt  xiaolinw@fzu.edu.cn}.}
	\and
	Lihua You\thanks{Corresponding author. School of Mathematical Sciences, South China Normal University, Guangzhou 510631, China,
		e-mail: {\tt ylhua@scnu.edu.cn}.}
}

\date{}

\makeatletter
\def\leftharpoonfill@{\arrowfill@\leftharpoonup\relbar\relbar}
\def\rightharpoonfill@{\arrowfill@\relbar\relbar\rightharpoonup}
\newcommand\rbjt{\mathpalette{\overarrow@\rightharpoonfill@}}
\newcommand\lbjt{\mathpalette{\overarrow@\leftharpoonfill@}}
\makeatother

\begin{document}
	\maketitle		
	
	\begin{abstract}
		Let $f(d)$ be the smallest value for which every bridgeless graph
		$G$ with diameter $d$ admits a strong orientation $\rbjt{G}$ such that the diameter of $\rbjt{G}$ is at most $f(d)$. Chv\'atal and Thomassen (JCTB, 1978) established general bounds for $f(d)$, and also  proved that $f(2)=6$ and $f(4)\geq 12$. The works of both Kwok, Liu and West (JCTB, 2010) and Wang and Chen (JCTB, 2022) together determined $f(3)=9$. In this paper, we improve the best known upper bound for $f(4)$ from $21$ (Babu et al., DAM, 2021) to \textbf{$18$}.
	\end{abstract}
	
	
	\section{Introduction}\label{sec-intro}
	
	Throughout  this paper, let $G=(V(G),E(G))$ be a connected bridgeless multigraph  without loops. 
	For any $u,v \in V(G)$, the distance $d(u,v)$ is the length of a shortest $(u,v)$-path in $G$, and the \emph{diameter} of $G$ is defined as $d(G)=\max\{d(u,v)\mid u,v\in V(G)\}$. A \emph{bridge} in a graph $ G $ is an edge whose deletion results in a disconnected graph. A graph  is called \emph{bridgeless} if it contains no bridges.
	
	An \emph{orientation} $\rbjt{G}$ of $G$ is a digraph obtained from $G$ by assigning a direction to each edge.  We say an orientation $\rbjt{G}$ is \emph{strong} if there exists a directed path between any two vertices in $\rbjt{G}$. The \emph{directed distance} $\partial(u,v)$ is the length of the shortest directed path from $u$ to $v$ in $\rbjt{G}$. If $\rbjt{G}$ is strong, then the diameter of $\rbjt{G}$ is defined as $d(\rbjt{G})=\max\{\partial(u,v)|u,v\in V(\rbjt{G})\}$. 
	
	In 1939, Robbins  \cite{HR} established  a well-known theorem:  a connected undirected graph admits a strong orientation if and only if it is bridgeless. Then the \emph{oriented diameter} of a bridgeless graph $G$ is defined by $$\rbjt{diam}(G)=\min\{d(\rbjt{G})\mid  \rbjt{G} \text{ is a strong orientation of }G\}.$$ 
	
	In recent decades, numerous researchers have studied the oriented diameter and sought sharp upper bounds for various prescribed graph parameters and special graph classes.
	Let $\gamma(G)$ denote the \emph{domination number} of $G$.
	Fomin et al. \cite{FF} first established an upper bound on the oriented diameter $\rbjt{\mathrm{diam}}(G)$ of a graph $G$, proving $\rbjt{\mathrm{diam}}(G)\leq 9\gamma(G)-5$. Subsequently, Kurz and L\"{a}tsch \cite{SK} improved this bound to $\rbjt{\mathrm{diam}}(G)\leq 4\gamma(G)$, and conjectured that $\rbjt{\mathrm{diam}}(G)\leq\left\lceil \frac{7\gamma(G)+1}{2} \right\rceil$. This conjecture was later confirmed by Wang and Chen \cite{WC}.  
	In terms of degree conditions, Dankelmann et al. \cite{PD} gave the tight bound $\rbjt{\mathrm{diam}}(G) \leq n - \Delta + 3$, where $\Delta$ denotes the maximum degree of $G$. For the minimum degree $\delta$, Bau and Dankelmann \cite{SB} derived $\rbjt{\mathrm{diam}}(G) \leq 11n/(\delta + 1) +9$, and Surmacs \cite{MS} further improved this result to $\rbjt{\mathrm{diam}}(G) < 7n/(\delta + 1)$. For other parameters, see 
	\cite{BC, RL}. 
	
	For special graph classes, Gutin \cite{GG} and Plesník \cite{JP} showed that the oriented diameter of any complete $k$-partite graph with $k\geq 3$ lies between $2$ and $4$, while \v{S}olt\'es \cite{LS} proved that the oriented diameter of every complete bipartite graph is at most $4$. Wang and Chen \cite{XW2} showed that all but four maximal outerplanar graphs of order $n \geq 3$ satisfy $\rbjt{\mathrm{diam}}(G) \leq \left\lceil \frac{n}{2} \right\rceil$.
	Some related research on graph oriented diameter is summarized in the survey \cite{KK}.
	A mixed graph is a graph that contains both undirected and directed edges. Related results on the oriented diameter of mixed graphs can be found in \cite{BD2,FR,ZD,HL1,HL2,HL3}.

	The most direct and interesting problem is: given the diameter of $G$, how small can its oriented diameter be?
	Moreover,
	let $f(d)$ be the smallest value for which every bridgeless graph $G$ with $d(G)=d$ admits a strong orientation $\rbjt{G}$ satisfying $d(\rbjt{G}) \leq f(d)$. In 1978, Chvátal and Thomassen \cite{VC} introduced $f(d)$, confirmed its existence, and obtained the following result. 
	
	\begin{theorem}[Chvátal and Thomassen \cite{VC}]\label{t1}
		Let $G$ be a bridgeless graph with $d(G)=d$. Then $$\frac12 d^2+d\leq f(d)\leq 2d^2+2d.$$
	\end{theorem}
	
	Recently, Babu et al. \cite{BD}  improved the upper bound to $f(d)\leq1.373d^2 +6.971d-1$, which is strictly smaller than $2d^2+2d$ for all integers $d\geq 8$. Since it is difficult to derive a general bound for $f(d)$, it is worthwhile to find the exact value of $f(d)$ for small $d$.
	In 1980, Boesch and Tindell \cite{FB} showed that $f(1)=3$.
	Chvátal and Thomassen \cite{VC}  proved that $f(2)=6$. 
	Kwok et al. \cite{PK} established $9\le f(3)\le 11$ in 2010; subsequently, Wang and Chen \cite{XW} further determined the exact value $f(3)=9$ in 2022.
	By Theorem \ref{t1},  $12\leq f(4)\leq 40$. Babu et al. \cite{BD} improved its upper bound to 
	$f(4)\leq 21$. In this paper, we prove the following main result.
	
	\begin{theorem}\label{t7}
		$f(4)\leq 18$.
	\end{theorem}

	Our proof is based on the previous results of Lin and You \cite{LY}.  They provided the following definition as an important tool to study $f(d)$.
	
	\begin{definition}[Lin and You  \cite{LY}]\label{def1}
		Let $G$ be a bridgeless graph and $e\in E(G)$. The \emph{edge girth} of $e$ is the length of a shortest cycle containing $e$ in $G$, denoted by $l_{G}(e)$, and the \emph{edge girth} of $G$ is defined as $ g^*(G) = \max \left\{ l_{G}(e) \mid e \in E(G) \right\}.$
	\end{definition}
	
	Using Definition \ref{def1}, Lin and You \cite{LY} introduced the notation:
	$$F(d,g^*) = \max\{\rbjt{\text{diam}}(G) \mid G \text{ is a bridgeless graph with } d(G) = d , g^*(G) = g^*\},$$
	and showed 
	$f(d) = \max\{F(d,g^*) \mid 2 \leq g^* \leq 2d+1\}$. 
	They also proved the following result.
	
	\begin{theorem}[Lin and You  \cite{LY}]\label{t5}
		$F(4,2)=4$, $F(4,9)=12$, $F(4,3)\leq 12$, $F(4,g^*)\leq 13$ for $g^*\in \{6,7,8\}$. 
	\end{theorem}
	
	In \cite{BD}, Babu et al. proved $f(4)\leq 21$, which implies $\max\{F(4,4),F(4,5)\}\leq 21$. Based on Theorem \ref{t5}, Lin and You \cite{LY} proposed the following problem.
	
	\begin{problem}[Lin and You  \cite{LY}]\label{prob1}
		Provide improved upper bounds for $F(4,4)$ and $F(4,5)$.
	\end{problem}	
	
	To address Problem~\ref{prob1}, we investigate bridgeless graphs $G$ with $d(G)=4$ and $g^*(G)\in\{4,5\}$. Our main result is stated below.
	
	\begin{theorem}\label{t6}
		Let $G$ be a bridgeless graph with $d(G)=4$ and $g^*(G)\in \{4,5\}$. Then $\rbjt{diam}({G})\leq g^*(G)+13$.
	\end{theorem}
	
	By Theorem \ref{t6}, we have  $F(4,4)\leq 17$ and $F(4,5)\leq 18$. Combining $f(d)=\max\{F(d,g^*)\mid 2\leq g^*\leq 2d+1\}$ and Theorem \ref{t5}, we obtain Theorem \ref{t7}.
	
	In the rest of this paper, we introduce some notations and some important lemmas in Section \ref{sec-pre}.
	The proof of Theorem \ref{t6} is postponed to Section \ref{sec-3}. 
	We propose several new orientation techniques and combine them with existing methods to investigate the oriented diameter of graphs with diameter 4. Our approach improves the best known upper bound from $21$ to $18$. The orientation techniques developed in the proof may also be useful in the study of related problems on oriented diameters.

	\section{Notations and Lemmas}\label{sec-pre}
	Since this paper deals with multigraphs, for any two vertices $u$ and $v$, the notation $uv$ denotes a specific edge between $u$ and $v$. 
	
	Let $u_1,u_2,\dots,u_k$ be distinct vertices. A \emph{directed path} is a sequence $u_1 u_2 \ldots u_k$ such that $u_1\rightarrow u_2\rightarrow\cdots \rightarrow u_k$; a \emph{directed cycle} is a sequence $u_1 u_2 \ldots u_ku_1$ such that $u_1\rightarrow u_2\rightarrow \cdots \rightarrow u_k\rightarrow u_1$. We say that $u_1u_2\cdots u_k$ is a \textit{bidirectional path} if either $u_1u_2\cdots u_k$ or $u_ku_{k-1}\cdots u_1$ is a directed path. Similarly, $u_1u_2\cdots u_ku_1$ is called a \textit{bidirectional cycle} if either $u_1u_2\cdots u_ku_1$ or $u_ku_{k-1}\dots u_1u_k$ is a directed cycle.
	
	Recall that a mixed graph is a graph containing both directed and undirected edges. In a mixed graph, a vertex $z$ is called \emph{undirected} if all edges incident to $z$ are undirected; otherwise, $z$ is called \emph{directed}. A cycle (path) $H$ in a mixed graph is called \emph{mixed} if each undirected edge of $H$ can be oriented to form a bidirectional cycle (path). Such an orientation is called a bidirectional-cycle orientation (respectively, a bidirectional-path orientation), abbreviated as a $B$-$C$ orientation (respectively, a $B$-$P$ orientation).

	Let $G$ be a bridgeless graph. For $S\subseteq V(G)$, we use $G[S]$ to denote the subgraph of $G$ induced by $S$. For $v\in V(G)$ and $U\subseteq V(G)$, let $N(v)=\{u\ |\ uv\in E(G)\}$ and $N(U)= \bigcup\limits_{v \in U} N(v) - U$. Let $U$ and $V$ be two subsets of $V(G)$. We use $[U,V]$ to denote the set of all edges between $U$ and $V$. If $U=\{u\}$ or $V=\{v\}$, we write $[u,V]$ or $[U,v]$ for short. The notation $U\rightarrow V$ means that every edge in $[U,V]$ is oriented from $U$ to $V$. We write $u\rightarrow V$ or $U\rightarrow v$ if $U=\{u\}$ or $V=\{v\}$. Let $ u \in V(G) $ and $ S \subseteq V(G) $ with $ u \notin S $. A $ (u, S) $-path is defined as the shortest path from the vertex $ u $ to any vertex in $ S $, and $ d(u, S) $ denotes the length of such a path. A directed $ (u, S) $-path is the shortest directed path from $ u $ to $ S $, and $ \partial(u, S) $ denotes the length of this directed path. Similarly, a directed $(S,u)$-path and its length $\partial(S,u)$ are also defined. Since $\partial(u,S)$ and $\partial(S,u)$ need not be equal, we define $ \theta(u, S) $ as the maximum of these two values, that is, $ \theta(u, S) = \max\{\partial(u, S), \partial(S, u)\} $. Similarly, we define $\theta(u,v)=\max\{\partial(u,v), \partial(v,u)\}$. 
	
	A special orientation was implicitly used by Kwok et al.\ in the proof of \cite[Lemma 3.4]{PK}. Wang and Chen~\cite{XW} later refined this orientation and introduced the term \emph{$R$--$S$ orientation}. We recall the definition below.
	\begin{definition}[\text{Kwok et al.} \cite{PK}, \text{Wang and Chen} \cite{XW}]\label{def2}
		Let $G$ be a graph, and let $ R $ and $ S $ be two disjoint vertex subsets of $G$ such that $ G[S] $ has no isolated vertex, and $ N(w) \cap R \neq \emptyset $ for every $ w \in S $. Let $F$ be a spanning forest of $G[S]$ with bipartition $(V_1,V_2)$. Orient all edges of $ [R, V_1] $, $ [V_1, V_2] $ and $ [V_2, R] $ as $ R \rightarrow V_1 \rightarrow V_2 \rightarrow R $. This orientation is called $R$-$S$ orientation $($see Figure \ref{fig4}$)$.
	\end{definition}
	
	\begin{figure}[ht]
		\centering
		\includegraphics[scale=0.18]{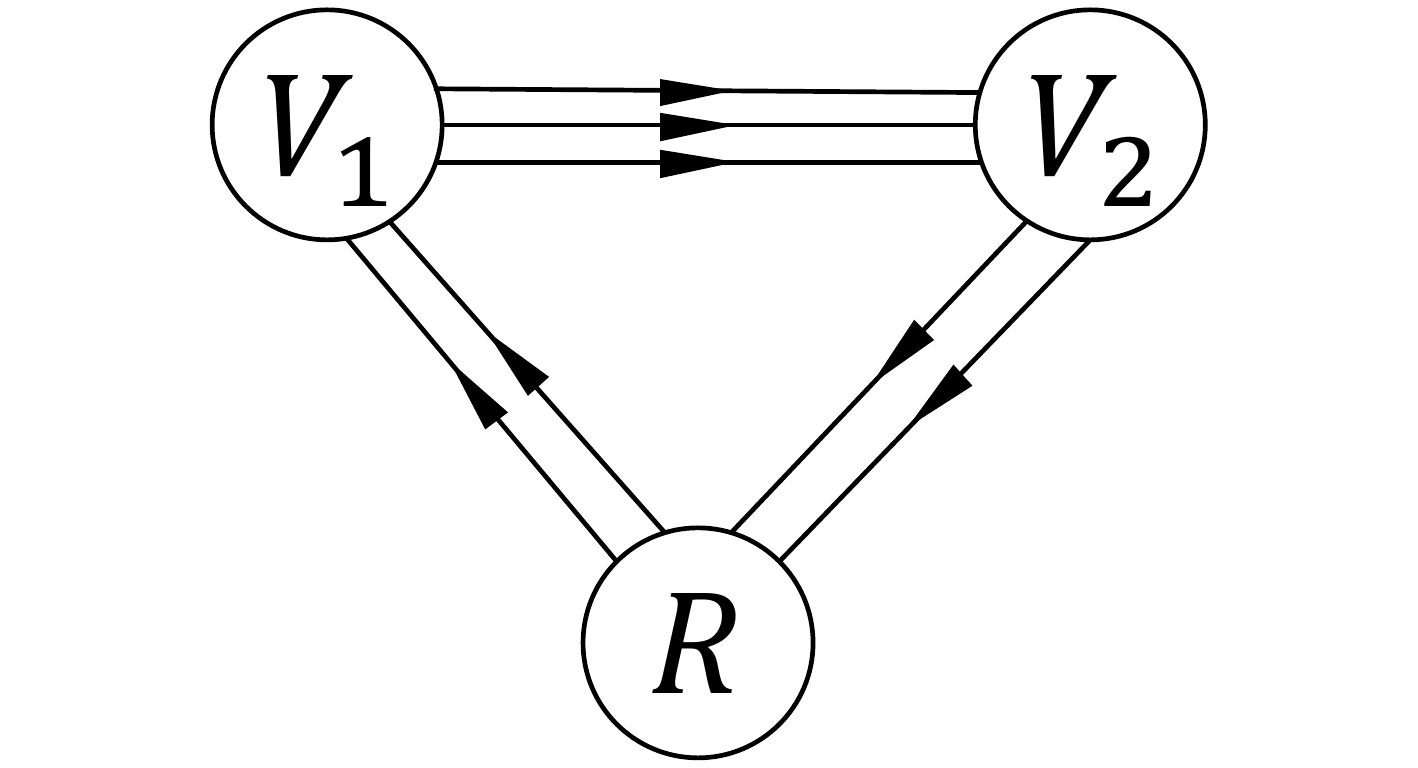}
		\vspace{-10pt}
		\caption{$R$-$S$ orientation.}\label{fig4}
	\end{figure}
	\begin{lemma}[Kwok et al. \cite{PK}]\label{l1}
		In an $R$-$S$ orientation, every vertex $w\in S$ satisfies $ \theta(w, R) \leq 2 $. 
	\end{lemma} 
	
	\section{The Proof of Theorem \ref{t6}}\label{sec-3}
	
	We establish the main result step by step: partition the vertex set of $G$ into several small subsets, prescribe an orientation for the edges among these subsets, and then calculate the directed distances between any ordered pairs of vertices.
	
	Let $G$ be a bridgeless graph with $d(G)=4$ and $g^*=g^*(G)\in\{4,5\}$. Then there is an edge $e=uv\in E(G)$ such that $l_G(e)=g^*\in\{4,5\}$. 
	
	Let $ S_{i,j} = \{ w \in V(G) \mid d(w, u) = i \text{ and } d(w, v) = j \} $. Then $S_{1,1}=\emptyset$ by $g^*\geq 4$, and thus $ S_{1,2}, S_{2,1}, S_{2,2}, S_{2,3}, S_{3,2}, S_{3,3} , S_{3,4}, S_{4,3}, S_{4,4}$ form a partition of $V(G)\backslash \{u,v\}$. We further partition every $S_{i,j}$ except $ S_{2,2} $ as shown in Figure \ref{fig1}.
	\begin{align*}
		A' &= \{ w \mid w \in S_{1,2} \text{ and } N(w) \subseteq S_{1,2} \cup \{u\} \}, & A &= S_{1,2} - A'; \\
		B' &= \{ w \mid w \in S_{2,1} \text{ and } N(w) \subseteq S_{2,1} \cup \{v\} \}, & B &= S_{2,1} - B'; \\
		I' &= \{ w \mid w \in S_{2,3} \text{ and } N(w) \subseteq S_{2,3} \cup A \}, & I &= S_{2,3} - I'; \\
		J' &= \{ w \mid w \in S_{3,2} \text{ and } N(w) \subseteq S_{3,2} \cup B \}, & J &= S_{3,2} - J'; \\
		X' &= \{ w \mid w \in S_{3,3} \text{ and } |[w, S_{2,2} \cup I \cup J\cup K\cup L]| = 1 \}, & X &= S_{3,3} - X';\\
		K' &= \{ w \mid w \in S_{3,4} \text{ and } N(w) \subseteq S_{3,4} \cup I \}, & K &= S_{3,4} - K'; \\
		L' &= \{ w \mid w \in S_{4,3} \text{ and } N(w) \subseteq S_{4,3} \cup J \}, & L &= S_{4,3} - L'; \\
		M' &= \{ w \mid w \in S_{4,4} \text{ and } |[w, S_{3,3} \cup K \cup L]| = 1 \}, & M &= S_{4,4} - M'.
	\end{align*}
	
	\vspace{-15pt}
	\begin{figure}[ht]
		
		\centering
		\includegraphics[scale=0.6]{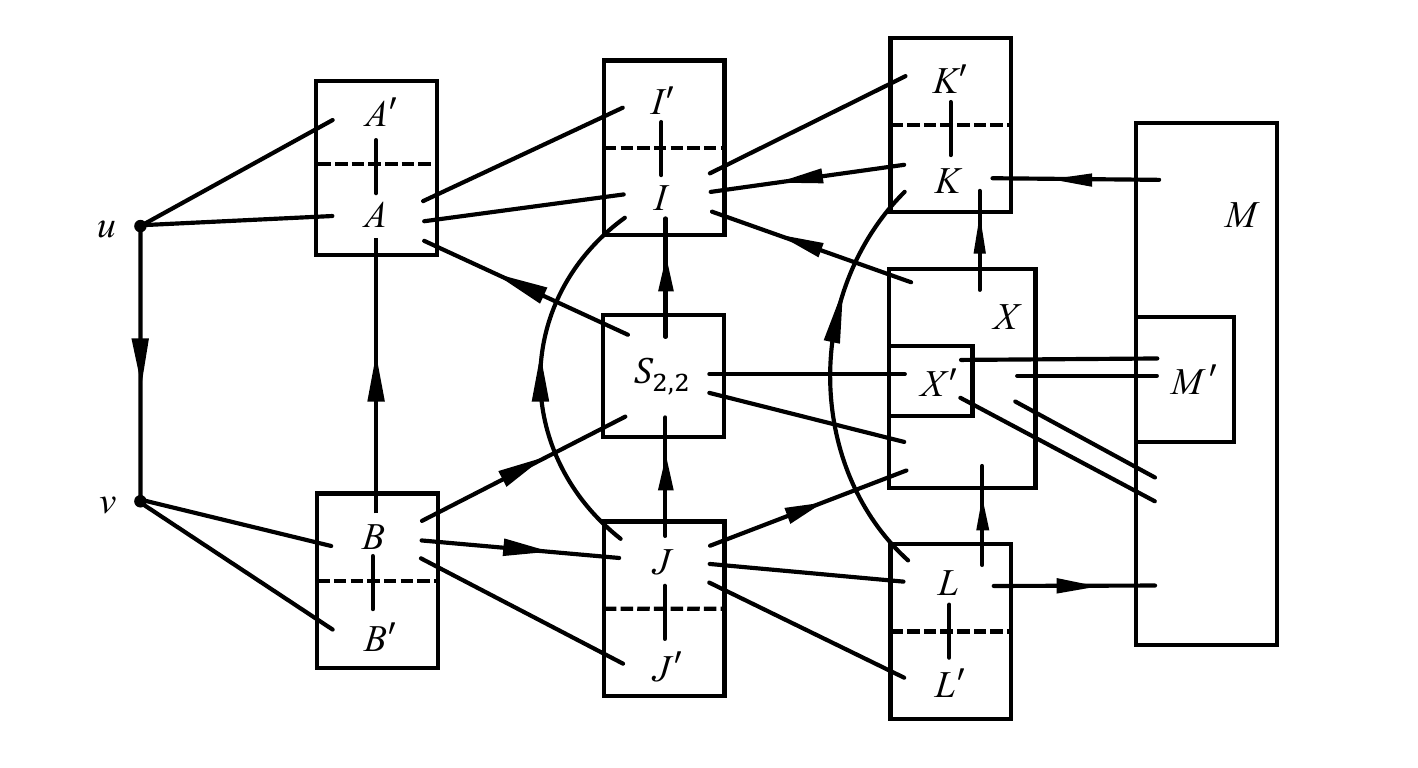}
		\vspace{-20pt}
		\caption{The partition of $V(G)$; orientations will be given subsequently.}\label{fig1}
		\vspace{-5pt}
	\end{figure}	
	
	\begin{proposition}\label{pro1}
		{\rm (i)} For $w\in S_{4,4}$, if $N(w)\cap S_{3,3}=\emptyset$, then $N(w)\cap K\neq \emptyset$ and $N(w)\cap L\neq \emptyset$; if $w\in M'$, then $|[w,S_{3,3}]|=1$ and $N(w)\cap (K\cup L)=\emptyset$.\\
		{\rm (ii)} For $w\in S_{3,3}$, if $N(w)\cap S_{2,2}=\emptyset$, then $N(w)\cap I\neq \emptyset$ and $N(w)\cap J\neq \emptyset$; if $w\in X'$, then $|[w,S_{2,2}]|=1$ and $N(w)\cap (I\cup J\cup K\cup L)=\emptyset$.
	\end{proposition}
	
	\begin{proof}
		If $N(w)\cap S_{3,3}= \emptyset$, then $N(w)\cap K\neq \emptyset$ and $N(w)\cap L\neq \emptyset$ by $d(w,u)=d(v,w)=4$.
	
		If $w\in M'$ and $N(w)\cap S_{3,3}$
		=$\emptyset$, then $|[w,S_{3,3}\cup K\cup L]|\geq 2$, a contradiction with the definition of $M'$. Thus $N(w)\cap S_{3,3}\neq\emptyset$, which implies $|[w,S_{3,3}]|=1$ and $N(w)\cap (K\cup L)=\emptyset$ by the definition of $M'$.  Similarly, we can prove (ii) holds.
	\end{proof}
	
	We orient $uv$ as $u \to v$; then $\partial(u,v) = 1$. Based on the forthcoming Orientations \ref{cons1}-\ref{cons11} and Claims \ref{cla1}–\ref{cla13}, the relevant directed distances are listed in Tables \ref{Table8-1}–\ref{Table8-2}. Note that $\partial(x,y) \leq \partial(x,u) + \partial(u,v) + \partial(v,y)$. 
	
	\begin{table}[H]
		\centering
		
		\renewcommand{\arraystretch}{1.1}
		\setlength{\tabcolsep}{4pt}
		\begin{tabular}{|l|*{8}{>{\centering\arraybackslash}p{1cm}|}}
			\hline
			\textbf{for \(w\) in} & \(A\) & \(B\)& $A'$ & $B'$ & \(I\) & \(J\)& $I'$&$J'$  \\ \hline
			\(\partial(w, u) \leq\) & $g^*-1$ & 9 &2&$g^*+1$& $g^*$ & 8&$g^*+2$ & 11 \\ \hline
			\(\partial(v, w) \leq\) & $g^*+3$ & $g^*-1$ & $g^*+1$&2&9 & 2 &$g^*+5$& $g^*+1$ \\ \hline
		\end{tabular}
		\vspace{-3pt}
		\caption{Directed-distance bounds, Part 1.}\label{Table8-1}
	\end{table}

	\begin{table}[H]
		\centering
	
		\vspace{5pt}
		\renewcommand{\arraystretch}{1.1}
		\setlength{\tabcolsep}{4pt}
		\begin{tabular}{|l|*{8}{>{\centering\arraybackslash}p{1cm}|}}
			\hline
			\textbf{for \(w\) in} & \(K\) & \(L\)&$K'$&$L'$&$S_{2,2}$& $S_{3,3}$ &\(M\)&$M'$  \\ \hline
			\(\partial(w, u) \leq\) &3 & 7& $g^*+2$&7&2&$g^*+1$ & $g^*+2$&$g^*+3$  \\ \hline
			\(\partial(v, w) \leq\) & 5 & 4 & $9$&4&2&$6$ &$7$&$8$  \\ \hline
		\end{tabular}
		\vspace{-5pt}
		\caption{Directed-distance bounds, Part 2.}\label{Table8-2}
	\end{table}

	\begin{claim}\label{cla18}
		If $\{x,y\}\subseteq A\cup A'\cup B'\cup I\cup K\cup L\cup K'\cup L'\cup S_{2,2}\cup  S_{3,3}\cup M\cup M'$ or $\{x,y\}\cap \{u,v\}\neq \emptyset$, then $\partial(x,y)\leq g^*+13$.
	\end{claim}
	
	\begin{proof}
		By Tables \ref{Table8-1}–\ref{Table8-2}, we have $\partial(x,u) \le g^* +3$ and $\partial(v,y) \le 9$. Hence $
		\partial(x,y)\leq \partial(x,u) + \partial(u,v) + \partial(v,y) \le  g^* +13$ by $\partial(u,v)=1$.
		The cases in which one of $x,y$ is $u$ or $v$ follow immediately from $\partial(u,v) = 1$, $\partial(v,u) = g^* -1$, and the same two tables.
	\end{proof}
	
	By Claim \ref{cla18}, it remains to consider the following claim. 
	
	\begin{claim}\label{cla19} Let $x,y \in V(G)$ with $\{x,y\} \cap \{u,v\} = \emptyset$. Then we have\\
		{\rm (i)} If $\{x,y\}\cap (B\cup J)\neq \emptyset$, then $\partial(x,y)\leq g^*+13$.\\
		{\rm (ii)} If $\{x,y\}\cap (I'\cup J')\neq \emptyset$, then $\partial(x,y)\leq g^*+13$.
	\end{claim}
	
	We postpone the proof of Claim \ref{cla19} to the end of this section. By combining Claims $\ref{cla18}$ and $\ref{cla19}$, we obtain $d(\rbjt{G}) \leq g^* + 13$. Consequently, $\rbjt{\text{diam}}(G) \leq d(\rbjt{G}) \leq g^* + 13$. This completes the proof of Theorem~\ref{t6}.
	{\hfill $\blacksquare$\par}

	To define Orientations \ref{cons1}-\ref{cons11} and prove Claims \ref{cla1}--\ref{cla13}, we further partition the sets $M,X',K,I,J,L,$ $ A,B$ as follows.
	\begin{align*}
		M_1&=\{w\mid w\in M \text{ and } N(w)\cap  K\neq\emptyset\}, \quad M_2=\{w\mid w\in M-M_1\text{ and } N(w)\cap L\neq\emptyset\},\\ M_3&=M-M_1-M_2;\\
		X'_1 &= \{ w \mid w \in X' \text{ and } N(w) \cap M_1 \neq \emptyset \},\quad X'_2= \{ w \mid w \in X'-X'_1 \text{ and } N(w) \cap M_2 \neq \emptyset \},\\
		X'_3&=\{w\mid w\in X'-X'_1-X'_2\text{ and } N(w)\cap X\neq\emptyset\},\\
		X'_4 &= \{ w \mid w \in X' - \cup_{i=1}^3X'_i \text{ and } w \text{ is isolated in } G[X' - \cup_{i=1}^3X'_i] \},\quad \\
		X'_5 &= \{ w \mid w \in X' - \cup_{i=1}^3X'_i \text{ and } w \text{ is not isolated in } G[X' - \cup_{i=1}^3X'_i] \};\\
		K_1 &= K \cap N(S_{3,3} \cup L),\quad K_2=K\cap N(K_1)-K_1, \quad  K_3 = K - K_1-K_2;\\
		I_1 &= I \cap N(S_{2,2} \cup J), \  I_2=I\cap N(S_{3,3})-I_1,\  I_3=I\cap N(I_1)-I_1-I_2,\\
		I_4 &=I\cap N(I_2\cup K_1)-\cup_{i=1}^3 I_i,\ I_5=I\cap N(K_2\cup K_3)-\cup_{i=1}^4 I_i, \  I_6= I-\cup_{i=1}^5 I_i;\\
		J_1 &= J \cap N(S_{2,2} \cup I), \  J_2=J\cap N(S_{3,3})-J_1,\ J_3=J\cap N(J_1)-J_1-J_2,\\
		L_1 &= L \cap N(K),\quad L_2=L\cap N(L_1\cup M_1)-L_1, \\ L_3 &= L\cap N(J_1)\cap N(L')-L_1-L_2,\quad L_4=L-\cup_{i=1}^3 L_i; \\
		J_4 &=J\cap N(J_2\cup L_1)-\cup_{i=1}^3 J_i,\    J_5=J\cap N(L_2\cup L_3\cup L_4)-\cup_{i=1}^4 J_i, \  J_6= J-\cup_{i=1}^5 J_i;\\
		A_1 &= A \cap N(B),\quad  A_2 = A \cap N(S_{2,2}) - A_1,\quad
		A_3 = A\cap N(A_1)-\cup_{i=1}^2 A_i,\\
		A_{3+j} &= A \cap N(I_{j}) - \cup_{i=1}^{2+j}A_i,\ j=1,2,\dots,5,\quad  A_{9}=A-\cup_{i=1}^8 A_i;\\
		B_1 &= B \cap N(A),\quad  B_2 = B \cap N(S_{2,2}) - B_1,\quad
		B_3 =B\cap N(B_1)-\cup_{i=1}^2 B_i ,\\
		B_{3+j} &= B \cap N(J_{j}) - \cup_{i=1}^{2+j}B_i,\ j=1,2,\dots,6,\quad  B_{10}=B-\cup_{i=1}^9 B_i.
	\end{align*}
	
	The definitions of $A_{9},B_{10},I_6,J_6,K_3,L_4,M_1,M_2$ give the following observation.
	
	\begin{observation}\label{obs1}
		For every $w\in A_{9}\cup B_{10}\cup I_6\cup J_6\cup K_3\cup L_4$, we have $N(w)\cap S_{2,2}=\emptyset$, $N(w)\cap W_1= \emptyset$ and $N(w)\cap W_2\neq\emptyset$, where $W_1,W_2$ are defined in Table~\ref{Table 2}.
		\vspace{0cm}
		\begin{table}[ht]
			\centering

			\renewcommand{\arraystretch}{1.1}
			\setlength{\tabcolsep}{6pt}
			\begin{tabular}{|c|c|c|c|c|c|c|}
				\hline
				when $w$ in & $A_{9}$ & $B_{10}$ & $I_6$ & $J_6$ & $K_3$&  $L_4$ \\ \hline
				$W_1=$ & $B\cup (\cup_{i=1}^5 I_i)$ & $ A\cup  J$ & $ J\cup S_{3,3}\cup K$ & $I\cup S_{3,3}\cup L$ & $ S_{3,3}\cup L$ & $K$ \\ \hline
				$W_2=$ & $ I_6\cup I'$ & $ J'$ & $ K'$ & $ L'$ & $ M_1$ & $ S_{3,3}\cup M$ \\ \hline
			\end{tabular}
			\vspace{-5pt}
			\caption{Neighborhoods of selected vertices.}
			\label{Table 2}
		\end{table}
	\end{observation}

    We now construct the desired orientation of $G$. During this process, the undirected graph $G$ is first transformed into a mixed graph, and then into a directed graph.
	\begin{Orientation}\label{cons1}
		Let $  u \rightarrow v $, $B\rightarrow A\cup J$, \( B\cup J \rightarrow S_{2,2} \rightarrow A\cup I  \), $J\rightarrow I\cup S_{3,3}$, $L\rightarrow S_{3,3}\cup K\cup M$, $ S_{3,3}\rightarrow K\cup I$, $M\rightarrow K\rightarrow I$, as shown in  Figure \ref{fig1}; Let $J\rightarrow L_1\cup L_2\cup L_4$, $J\rightarrow L'\cap N(L_3)$, $L'\cup (J-J_1)\rightarrow L_3\rightarrow J_1$,  $\cup_{i=1}^5 I_i\rightarrow  \cup_{i=1}^8 A_i\rightarrow u$, $v\rightarrow \cup_{i=1}^9 B_i$, \( A_i \rightarrow A_j \) for \( 1 \leq i < j \leq 8\), \( B_j \rightarrow B_i \) for \( 1 \leq i < j \leq 9 \), \( I_i \rightarrow I_j \) for \( 1 \leq i < j \leq 5 \), \(J_j \rightarrow J_i \) for \( 1 \leq i < j \leq 5 \), $K_i\rightarrow K_j$ for $1\leq i<j\leq 3$, and $L_j\rightarrow L_i$ for $1\leq i<j\leq 4$.
		
	\end{Orientation}

	\begin{samepage}
		\begin{claim}\label{cla1}
			{\noindent\rm (i)} $\partial(u,v)=1$, $\partial(v,u)= g^*-1$, $\partial(w,u)=\partial(v,w)=2$ for $w\in S_{2,2}$.\\
			{\rm (ii)} Table \ref{Table1} presents upper bounds on directed distances for $I\cup A\cup J\cup B$.
			\begin{table}[htbp]
				\centering
		
				\renewcommand{\arraystretch}{1.1}
				\setlength{\tabcolsep}{2pt}
				$
				\begin{array}{|*{18}{c|}} 
					\hline
					\text{for $w$ in}
					&I_1&I_3
					&A_1&A_2&A_3&A_4&A_6
					&J_1&J_3
					&B_1&B_2&B_3&B_4&B_6\\
					\hline
					\partial(w,u)\le
					&2&2&1&1&1&1&1&3&4&2&3&3&4&5\\
					\hline
					\partial(v,w)\le
					&3&4&2&3&3&4&5&2&2&1&1&1&1&1\\
					\hline
				\end{array}
				$
				\vspace{-5pt}
				\caption{Directed-distance bounds for $w\in I\cup A\cup J\cup B$}
				\label{Table1}
			\end{table}
			\vspace{-10pt}
		\end{claim}
		
		\begin{proof}
			By Orientation \ref{cons1}, (i) holds since $u\rightarrow v$, $v\rightarrow B_1\rightarrow A_1\rightarrow u$ appears only if $g^*=4$, and $v\rightarrow B_1\cup B_2\rightarrow S_{2,2}\rightarrow A_1\cup A_2\rightarrow u$. 
			
			For $w\in I_1$, we have $N(w)\cap (S_{2,2}\cup J_1)\neq \emptyset$, which implies $\partial(w,u)\leq 2$ and $\partial(v,w)\leq 3$ since $v\rightarrow \cup_{i=1}^4 B_i\rightarrow S_{2,2}\cup J_1\rightarrow w\rightarrow \cup_{i=1}^4 A_i\rightarrow u$ under Orientation \ref{cons1}. 
			
			For $w\in I_3$, $N(w)\cap I_1\neq\emptyset$, which implies $\partial(w,u)\leq 2$ and $\partial(v,w)\leq \partial(v,I_1)+\partial(I_1,w)\leq 4$ since $I_1\rightarrow w\rightarrow \cup_{i=1}^6 A_i\rightarrow u$ under Orientation \ref{cons1}.
			
			For $w\in A_1\cup A_2\cup A_3\cup A_4\cup A_6$, we have $\partial(w,u)=1$ since $w\rightarrow u$ under Orientation \ref{cons1}. If $w\in A_1$, then $N(w)\cap B_1\neq\emptyset$, which implies $\partial(v,w)\leq 2$ by $v\rightarrow B_1\rightarrow A_1$. If $w\in A_2$, then $N(w)\cap S_{2,2}\neq\emptyset$, which implies $\partial(v,w)\leq \partial(v,S_{2,2})+\partial(S_{2,2},w)\leq 3$ by (i) and $S_{2,2}\rightarrow w$.	If $w\in A_3$, then $\partial(v,w)\leq \partial(v,A_1)+\partial(A_1,w)\leq 3$ by $N(w)\cap A_1\neq\emptyset$ and $A_1\rightarrow w$. If $w\in A_4$, then $\partial(v,w)\leq \partial(v,I_1)+\partial(I_1,w)\leq 4$ by $I_1\rightarrow w$. If $w\in A_6$, then $\partial(v,w)\leq \partial(v,I_3)+\partial(I_3,w)\leq 5$ by $I_3\rightarrow w$.
			
			Similarly, for $w\in J_1\cup J_3\cup  B_1\cup B_2\cup B_3\cup B_4\cup B_6$, the conclusion holds.
		\end{proof}
	\end{samepage}


	\begin{Orientation}\label{cons4}
	Let $w \in X$. If $|[w, S_{2,2}]| \geq 2$, we orient the edges of $[w, S_{2,2}]$ in two ways. If $|[w, S_{2,2}]| =1$, let: \\{\rm (i)} $w \rightarrow S_{2,2}$ if $N(w) \cap J \neq \emptyset$, \\{\rm (ii)} $S_{2,2} \rightarrow w$ if $N(w) \cap J = \emptyset$ and $N(w) \cap I \neq \emptyset$, \\{\rm (iii)}  $S_{2,2}\rightarrow w\rightarrow K\cup M_1$ if $N(w)\cap (K\cup M_1)\neq\emptyset$ and $N(w)\cap (I\cup J)=\emptyset$, \\{\rm (iv)} $w\rightarrow S_{2,2}$ if $N(w)\cap (I\cup J\cup K\cup M_1)=\emptyset$.
	\end{Orientation}
	

\begin{claim}\label{cla4}
For every $w\in X$, we have $|[w, S_{2,2} \cup I \cup J\cup K\cup L]|\geq 2$, and 
\\{\rm(i)} $\partial(w,u)\leq 3$ and $\partial(v,w)\leq 3$ if $|[w,S_{2,2}]|=0$, or $|[w,S_{2,2}]|\geq 2$, or $N(w)\cap (I\cup J)\neq\emptyset$;
\\{\rm(ii)} $\partial(w,u)\leq 5$ and $\partial(v,w)\leq 3$ if $|[w,S_{2,2}]|=1$, $N(w)\cap (K\cup M_1) \neq\emptyset$, $N(w)\cap (I\cup J)=\emptyset$;
\\{\rm(iii)} $\partial(w,u)\leq 3$ and $\partial(v,w)\leq 5$ if $|[w,S_{2,2}]|=1$ and $N(w)\cap (I\cup J\cup K\cup M_1)=\emptyset$.
\end{claim}

\begin{proof}
By the definition of $X$, $|[w, S_{2,2} \cup I \cup J\cup K\cup L]|\geq 2$.

If $|[w,S_{2,2}]|=0$, then $N(w)\cap (I_1\cup I_2)\neq\emptyset$ and $N(w)\cap (J_1\cup J_2)\neq\emptyset$ by Proposition \ref{pro1}, and thus $v\rightarrow (B-B_{10})\rightarrow J_1\cup J_2\rightarrow w\rightarrow I_1\cup I_2\rightarrow (A-A_9) \rightarrow u$ under Orientation \ref{cons1}. Hence $\partial(v,w)\leq 3$ and $\partial(w,u)\leq 3$. 

If $|[w, S_{2,2}]| \geq 2$, by Orientations \ref{cons1}-\ref{cons4}, there exists $w_1,w_2\in S_{2,2}$ such that $w_1\rightarrow w\rightarrow w_2$, and thus $\partial(v,w)\leq \partial(v,w_1)+\partial(w_1,w)= 3$ and $\partial(w,u)\leq \partial(w,w_2)+\partial(w_2,u)=3$ by (i) of Claim \ref{cla1}. It remains to consider the case $|[w, S_{2,2}]|=1$, where the argument follows from Orientations \ref{cons1}-\ref{cons4}.

If $N(w)\cap (I\cup J)\neq\emptyset$, for $N(w)\cap J\neq\emptyset$, we have $v\rightarrow (B-B_{10})\rightarrow J_1\cup J_2\rightarrow w\rightarrow S_{2,2}$, $\partial(v,w)\leq 3$ and $\partial(w,u)\leq 3$; for $N(w)\cap I\neq\emptyset$ and $N(w)\cap J=\emptyset$, we have $S_{2,2}\rightarrow w\rightarrow I_1\cup I_2\rightarrow (A-A_9)\rightarrow u$, $\partial(v,w)\leq 3$ and $\partial(w,u)\leq 3$. 

If $N(w)\cap (K\cup M_1)\neq\emptyset$ and $N(w)\cap (I\cup J)=\emptyset$, we have $S_{2,2}\rightarrow w\rightarrow K\cup M_1$, $M_1\rightarrow K\rightarrow (I-I_6)\rightarrow (A-A_9)\rightarrow u$, and thus $\partial(v,w)\leq 3$ and $\partial(w,u)\leq 5$ by (i) of Claim \ref{cla1}. It remains to assume $N(w)\cap (I\cup J\cup K\cup M_1)=\emptyset$. Then $N(w)\cap L\neq\emptyset$ by $|[w, S_{2,2} \cup I \cup J\cup K\cup L]|\geq 2$. If $N(w)\cap (L_1\cup L_2\cup L_4)\neq\emptyset$, then $v\rightarrow (B-B_{10})\rightarrow J\rightarrow L_1\cup L_2\cup L_4\rightarrow w\rightarrow S_{2,2}$, which implies $\partial(w,u)\leq 3$ and $\partial(v,w)\leq 4$; if $N(w)\cap L_3\neq\emptyset$, then $v\rightarrow (B-B_{10})\rightarrow J\rightarrow L'\cap N(L_3)\rightarrow L_3\rightarrow w\rightarrow S_{2,2}$, which implies $\partial(w,u)\leq 3$ and $\partial(v,w)\leq 5$.
\end{proof}

\begin{proposition}\label{prop1}
Let $w\in X'_4$ such that $N(w)\cap (\cup_{i=1}^3 X'_i)=\emptyset$.  \\
{\rm (i)} Then  $N(w)\cap S_{2,2}\neq\emptyset$ and $N(w)\cap (M'\cup M_3)\neq\emptyset$.  \\      
{\rm (ii)} Let $w_0\in N(w)\cap S_{2,2}$, $w_1\in N(w)\cap(M'\cup M_3)$ with $|N(w_1)\cap S_{3,3}|$ maximized. If $|N(w_1)\cap S_{3,3}|=1$, then $g^*=5$, and the shortest cycle through $w_0w$ is the $5$-cycle $w_0wy_1y_2y_3w_0$, where $y_1\in M'\cup M_3$, $y_2\in S_{4,4}$, $y_3\in S_{3,3}$.
\end{proposition}

\begin{proof} (i) By Proposition \ref{pro1}, we have  $|[w,S_{2,2}]|=1$ and $N(w)\cap (I\cup J\cup K\cup L)=\emptyset$. Since $w\in X'_4$ and $N(w)\cap (\cup_{i=1}^3 X'_i)=\emptyset$, $N(w)\cap (S_{3,3}\cup M_1\cup M_2)=\emptyset$. Thus $\emptyset\neq N(w)-S_{2,2}\subseteq M'\cup M_3$ since $G$ is bridgeless. 
	
	(ii) For any $x\in N(w)\cap (M'\cup M_3)$, we have $N(x)\cap S_{3,3}=\{w\}$ by the maximality of $|N(w_1)\cap S_{3,3}|$, and $N(x)\cap (K\cup L)=\emptyset$.  Therefore, the shortest cycle containing $w_0w$ is a $5$-cycle $C=w_0wy_1y_2y_3w_0$ by $g^*\in\{4,5\}$, where $y_1\in M'\cup M_3$, $y_2\in S_{4,4}$ and $y_3\in S_{3,3}$, which implies $g^*=5$.
\end{proof}

By (ii) of Proposition \ref{pro1} and the definition of $X'$, we have $|N(w)\cap S_{2,2}|=1$ for any $w\in X'$. By the definition of $X'_5$, $G[X'_5]$ has no isolated vertices. We now orient some edges incident to $X'$ as follows.

\begin{Orientation}\label{cons5}
{\rm (i)}	Let $M_2\rightarrow X'_2\rightarrow S_{2,2}\rightarrow X'_1\rightarrow M_1$, $X\rightarrow X'_3\rightarrow S_{2,2}$, $R=S_{2,2}$ and $S=X'_5$. We then give an $R$-$S$ orientation for the edges of $G[S]$ and $[R,S]$. 

{\noindent\rm (ii)} Let  $w\in X'_4$ be an undirected vertex. Using the following iterative procedure, we assign orientations to some edges of $G$. If $N(w)\cap (\cup_{i=1}^3 X'_i)\neq\emptyset$, then let $X'_1\rightarrow w\rightarrow S_{2,2}$ if $N(w)\cap X'_1\neq\emptyset$, and let $ S_{2,2}\rightarrow w\rightarrow X'_2\cup X'_3$ otherwise. It remains to assume $N(w)\cap (\cup_{i=1}^3 X'_i)=\emptyset$. By Proposition \ref{prop1}, let $w_0\in N(w)\cap S_{2,2}$, and $w_1\in N(w)\cap(M'\cup M_3)$ such that $|N(w_1)\cap S_{3,3}|$ is maximum.

If $|N(w_1)\cap S_{3,3}|\geq 2$, then we take $w_2\in N(w_1)\cap S_{3,3}-\{w\}$ and $w_3\in N(w_2)\cap (S_{2,2}\cup I\cup J)$ such that $w_1w_2w_3$ is a mixed path by iteration. Since $w$ is undirected, we have $w_0ww_1w_2w_3$ is a mixed path or cycle, and thus we orient the edges of $w_0ww_1w_2w_3$ by a $B$-$C$ or $B$-$P$ orientation.

If $|N(w_1)\cap S_{3,3}|= 1$, then the shortest cycle containing $w_0w$ is a $5$-cycle $C=w_0wy_1y_2y_3w_0$, where $y_1\in M'\cup M_3$, $y_2\in S_{4,4}$, $y_3\in S_{3,3}$ by Proposition \ref{prop1}. If $y_1y_2$ is oriented, there exists a bidirectional path or cycle 
$w'_0w'y_1y_2y'_3y'_4$, where $w'\in S_{3,3}$, which implies $w'=w$ and $w$ is a directed vertex since $|N(y_1)\cap S_{3,3}|\leq |N(w_1)\cap S_{3,3}|=1$, a contradiction. Hence $w_0wy_1y_2$ is undirected since $w$ is undirected. If $C$ is mixed, then orient the edges of $C$ by a $B$-$C$ orientation. Otherwise, $y_2y_3$ is oriented, and there exists a bidirectional path 
$y_2y_3y''_4$ with $y''_4\in S_{2,2}\cup I\cup J$, yielding a mixed path or cycle
$C'=w_0wy_1y_2y_3y''_4$. Orient the edges of $C'$ by a $B$-$P$ or $B$-$C$ orientation.
\end{Orientation}

The directed graph obtained from Orientation \ref{cons5} is denoted by $\rbjt{D_1}$.

\begin{claim}\label{cla5}
{\rm (i)} For any $w\in X'_1$, $\partial(w,u)\le 5$ and $\partial(v,w)\le 3$;\\
{\rm (ii)} For any $w\in X'_2\cup X'_3$, $\partial(w,u)\le 3$ and $\partial(v,w)\le 6$;\\
{\rm (iii)} For any $w\in X'_4$, $\partial(w,u)\le g^*+1$ and $\partial(v,w)\le g^*+1$;\\
{\rm (iv)} For any $w\in X'_5$, $\partial(w,u)\le 4$ and $\partial(v,w)\le 4$.
\end{claim}

\begin{proof}
By (ii) of Proposition \ref{pro1}, for any $w\in X'$, $N(w)\cap S_{2,2}\neq\emptyset$.

For $i\in \{1,2\}$ and $w\in X'_i$, we have $N(w)\cap M_i\neq\emptyset$. For $w\in X'_3$, we have $N(w)\cap X\neq\emptyset$. By Orientations \ref{cons1} and \ref{cons5}, we have $v\rightarrow B_1\cup B_2\rightarrow S_{2,2}\rightarrow X'_1\rightarrow M_1\rightarrow K\rightarrow (I-I_6)\rightarrow (A-A_9)\rightarrow u$, which implies (i) holds. In addition, we have $v\rightarrow (B-B_{10})\rightarrow J\rightarrow (L-L_3)\rightarrow M_2\rightarrow X'_2\rightarrow S_{2,2}\rightarrow A_1\cup A_2\rightarrow u$, and $v\rightarrow (B-B_{10})\rightarrow J\rightarrow L'\cap N(L_3)\rightarrow L_3\rightarrow M_2\rightarrow X'_2$. By classifying the neighbors of $N(X'_2)\cap M_2$ as lying in $L- L_3$ or in $L_3$, we obtain the result for $X_2'$.  Similarly, $X\rightarrow X'_3\rightarrow S_{2,2}$ and Claim \ref{cla4} settle the case $w\in X'_3$. Hence (ii) holds. By Lemma \ref{l1} and Orientation \ref{cons5}, (iv) holds. 

Let $w\in X'_4$. If $N(w)\cap (\cup_{i=1}^3 X'_i)\neq\emptyset$, then the desired bounds follow by Orientation \ref{cons5} and (i)-(ii). It remains to assume $N(w)\cap (\cup_{i=1}^3 X'_i)=\emptyset$. 

Let $w_0\in S_{2,2}$ and $w_1\in M'\cup M_3$ be defined under Orientation \ref{cons5}.

If $|N(w_1)\cap S_{3,3}|\geq 2$, then by Orientation \ref{cons5}, there exists a bidirectional path (or cycle) $w_0ww_1w_2w_3$ such that $w_2\in S_{3,3}$ and $w_3\in S_{2,2}\cup I\cup J$. Moreover, $w_3\in S_{2,2}\cup I_1\cup I_2\cup J_1\cup J_2$. Thus $v\rightarrow B_1\cup B_2\rightarrow w_0\rightarrow w\rightarrow w_1\rightarrow w_2\rightarrow w_3\rightarrow \cup_{i=1}^5 A_i\rightarrow u$ when $w_3\in S_{2,2}\cup I_1\cup I_2$, and $v\rightarrow \cup_{i=1}^5 B_i\rightarrow  w_3\rightarrow w_2\rightarrow w_1\rightarrow w\rightarrow w_0\rightarrow A_1\cup A_2\rightarrow u$ when $w_3\in J_1\cup J_2$ by Orientations \ref{cons1} and \ref{cons5}, which implies $\partial(w,u)\leq 5$ and $\partial(v,w)\leq 5$. 

If $|N(w_1)\cap S_{3,3}|=1$, then $g^*=5$ by (ii) of Proposition \ref{prop1}, and there exists a bidirectional path (or cycle) $w_0wy_1y_2y_3y_4$ such that $y_3\in S_{3,3}$ and $y_4\in S_{2,2}\cup I\cup J$ by Orientation \ref{cons5}. Similar to the case $|N(w_1)\cap S_{3,3}|\geq 2$, we obtain $\partial(w,u)\leq 6$ and $\partial(v,w)\leq 6$.
\end{proof}
\vspace{10pt}
To orient the edges incident to $M\cup M'$, we partition $M_3,M'$ as follows.
\begin{align*}
M_{31}&=M_3\cap V(\rbjt{D_1}),\quad M_{32}=M_3-M_{31};\\
M'_1&=M'\cap V(\rbjt{D_1}),\quad M'_2=M'-M'_1;\\
M'_{21} &= \{ w \mid w \in M'_2 \text{ and } N(w) \cap M \neq \emptyset \},\quad \\
M'_{22} &= \{ w \mid w \in M'_2 - M'_{21} \text{ and } w \text{ is isolated in } G[M'_2 - M'_{21}] \},\quad \\
M'_{23} &= \{ w \mid w \in M'_2 - M'_{21} \text{ and } w \text{ is not isolated in } G[M'_2 - M'_{21}] \}.
\end{align*}
For any $w\in M_{32}$, we have $N(w)\cap {(K\cup L)}=\emptyset$, and thus $|[w,S_{3,3}]|\geq 2$ by the definition of $M'$.
By Proposition \ref{pro1} and the definition of $M'_{23}$, we have $N(w)\cap S_{3,3}\neq\emptyset$ for any $w\in M'_{23}$, and $G[M'_{23}]$ has no isolated vertices. We now orient some edges incident to $M'$.

\begin{Orientation}\label{cons6}
Let $M_2\rightarrow X\cup X'_1$, $ X\rightarrow M_1$, $M\rightarrow M'_{21}\rightarrow S_{3,3}$, $S_{3,3}\rightarrow M'_{22}\rightarrow M'_1\cup M'_{21}$, $R=S_{3,3}$, $S=M'_{23}$, $w \in M_{32}$. We orient the edges of $[w, S_{3,3}]$ in two ways, and give an $R$-$S$ orientation for the edges of $G[S]$ and $[R,S]$. 
\end{Orientation}

\begin{claim}\label{cla6}
Table \ref{Table3-1} presents upper bounds on directed distances for $L\cup M\cup M'$.
\begin{table}[htbp]
\centering
\renewcommand{\arraystretch}{1.1}
\setlength{\tabcolsep}{2pt}
$
\begin{array}{|*{12}{c|}} 
	\hline
	\text{for $w$ in}
	&L_1&L_2&L_3&L_4&M_1&M_2&M_{31}\cup M'_1&M_{32}& M'_2\\
	\hline
	\partial(w,u)\le
	&4&5&4&7&4&6&g^*&g^*+2&g^*+3\\
	\hline
	\partial(v,w)\le
	&3&3&4&3&4&5&g^*&7&8\\
	\hline
\end{array}
$\vspace{-5pt}
\caption{Directed-distance bounds for $w\in L\cup M\cup M'$}
\label{Table3-1}
\end{table}
\vspace{-12pt}
\end{claim}

\begin{proof} For $w\in L- L_4$, we have:
	$N(w)\cap K\neq\emptyset$ if $w\in L_1$,
	$N(w)\cap (L_1\cup M_1)\neq\emptyset$ if $w\in L_2$,
	and $N(w)\cap J_1\neq\emptyset,\,N(w)\cap L'\neq\emptyset$ if $w\in L_3$. Then $v\rightarrow (B-B_{10})\rightarrow J\rightarrow L_1\cup L_2$, $L_2\rightarrow L_1\cup M_1\rightarrow K\rightarrow (I-I_6)\rightarrow (A-A_{9})\rightarrow u$, and $v\rightarrow (B-B_{10})\rightarrow J\rightarrow L'\cap N(L_3)\rightarrow L_3\rightarrow J_1\rightarrow I_1\rightarrow (A-A_9)\rightarrow u$ by Orientation \ref{cons1}. Thus $\partial(v,w)\leq 3$, $\partial(w,u)\leq 4$ if $w\in L_1$, $\partial(v,w)\leq 3$, $\partial(w,u)\leq 5$ if $w\in L_2$, and $\partial(v,w)\leq 4,\partial(w,u)\leq 4$ if $w\in L_3$.

For $w\in L_4$, we have $v\rightarrow (B-B_{10})\rightarrow J\rightarrow L_4$ by Orientation \ref{cons1}, which implies $\partial(v,w)\leq 3$. By Observation \ref{obs1}, (ii) of Proposition \ref{pro1} and Orientation \ref{cons1}, we have $N(w)\cap (X\cup M_2)\neq\emptyset$ and $w\rightarrow X\cup M_2$. If $N(w)\cap X\neq\emptyset$, then $\partial(w,u)\leq \partial(w,X)+\partial(X,u)\leq 1+5=6$ by Claim \ref{cla4}. If $N(w)\cap M_2\neq\emptyset$, then $N(w')\cap K=\emptyset$ for any $w'\in M_2$, and thus $N(w')\cap S_{3,3}\neq\emptyset$ by Proposition \ref{pro1}, which implies $N(w')\cap (X\cup X'_1\cup X'_2)\neq\emptyset$.  By Orientation \ref{cons5} and \ref{cons6}, we have $w\rightarrow M_2\rightarrow X\cup X'_1\cup X'_2$, which implies $\partial(w,u)\leq \partial(w,w')+\partial(w',u)\leq 1+6=7$ by Claims \ref{cla4} and \ref{cla5}.

For $w\in M_1$, we have $N(w)\cap K\neq\emptyset$. If $N(w)\cap S_{3,3}=\emptyset$, then $N(w)\cap L\neq\emptyset$ by Proposition \ref{pro1}, and thus $N(w)\cap (L_1\cup L_2)\neq\emptyset$ by the definitions of $L_1$ and $L_2$. If $N(w)\cap S_{3,3}\neq\emptyset$, then $N(w)\cap (X\cup X'_1)\neq\emptyset$. Hence $ L_1\cup L_2\cup X\cup X'_1\rightarrow w\rightarrow K\rightarrow (I-I_6)\rightarrow (A-A_{9})\rightarrow u$ by Orientations \ref{cons1}, \ref{cons5} and \ref{cons6}, which implies $\partial(w,u)\leq 4$ and $\partial(v,w)\leq 4$ by the result of $w\in L_1\cup L_2$, Claims \ref{cla4} and \ref{cla5}.

For $w\in M_2$, we have $N(w)\cap (X\cup X'_1\cup X'_2)\neq\emptyset$ by the proof of the case $w\in L_4$. By Orientation \ref{cons1}, \ref{cons5} and \ref{cons6}, we have $L\rightarrow M_2\rightarrow X\cup X'_1\cup X'_2$. Then $\partial(w,u)\leq 6$ and $\partial(v,w)\leq 5$ by Claims \ref{cla4}, \ref{cla5} and the result of $w\in L$.

For $w\in M_{31}\cup M'_1$, Orientation \ref{cons5} yields a directed path (or cycle) $x_1x_2wx_3x_4x_5$ (only if $g^*=5$), $x_1x_2x_3wx_4x_5$ (only if $g^*=5$), or $y_1y_2wy_3y_4$ with $x_1,y_1\in S_{2,2}\cup J_1\cup J_2$, $x_4,y_3\in S_{3,3}$, $x_5,y_4\in S_{2,2}\cup I_1\cup I_2$, and thus Orientation \ref{cons1} gives the desired result. 

For $w\in M_{32}$, by Orientation \ref{cons6}, there exist $w_1,w_2\in S_{3,3}$ such that $w_1\rightarrow w\rightarrow w_2$. By Claims \ref{cla4} and \ref{cla5}, we have $\partial(x,u)\leq g^*+1$ and $\partial(v,x)\leq 6$ for any $x\in S_{3,3}$. Then the conclusion holds.

For $w\in M'_2$, we have $|[w,S_{3,3}]|=1$ and $N(w)\cap (K\cup L)=\emptyset$ by (i) of Proposition \ref{pro1}. If  $w\in M'_{21}$, then $N(w)\cap M\neq\emptyset$. This gives $M\rightarrow M'_{21}\rightarrow S_{3,3}$ under Orientation \ref{cons6}. From Claims \ref{cla4}, \ref{cla5} and $M=M_1\cup M_2\cup M_{31}\cup M_{32}$, we get $\partial(w,u)\leq g^*+2$ and $\partial(v,w)\leq 8$. If $w\in M'_{22}$, then $N(w)\cap (M'_1\cup M'_{21})\neq\emptyset$ since $G$ is bridgeless, and thus $\partial(w,u)\leq g^*+3$ and $\partial(v,w)\leq 7$ by Claims \ref{cla4}, \ref{cla5} and Orientation \ref{cons6}. If $w\in M'_{23}$, then $\partial(w,u)\leq g^*+3$ and $\partial(v,w)\leq 8$ by Claims \ref{cla4}, \ref{cla5} and Lemma \ref{l1}.
\end{proof}

\begin{claim}\label{cla7}
Table \ref{Table3} presents upper bounds on directed distances for $K\cup L\cup I\cup J\cup A\cup B$.

\begin{table}[htbp]
\centering
\setlength{\tabcolsep}{1pt}
\renewcommand{\arraystretch}{1.1}
$\begin{array}{|*{19}{c|}}  
	\hline
	\text{for $w$ in}
	&K_1&K_2&K_3&I_2&I_4&I_5&J_2&J_4&J_5&A_5&A_7&A_8&B_5&B_7&B_8\\
	\hline
	\partial(w,u)\le
	&3&3&3&2&2&2&4&5&8&1&1&1&5&6&9\\
	\hline
	\partial(v,w)\le
	&4&5&5&4&5&6&2&2&2&5&6&7&1&1&1\\
	\hline
\end{array}
$\vspace{-5pt}
\caption{Upper bounds on directed distances for some $w\in K\cup L\cup I\cup J\cup A\cup B$.}
\label{Table3}
\vspace{-12pt}
\end{table}
\end{claim}
\begin{proof}
By Orientation \ref{cons1}, we have $\partial(w,u)\leq 3$ if $w\in K$, $\partial(w,u)\leq 2$ if $w\in I_2\cup I_4\cup I_5$, $ \partial(v,w)\leq 2$ if $w\in J_2\cup J_4\cup J_5$, $\partial(w,u)\leq 1$ if $w\in A_5\cup A_7\cup A_8$, and $\partial(v,w)\leq 1$ if $w\in  B_5\cup B_7\cup B_8$.

For $w\in K_1$, we have $N(w)\cap (S_{3,3}\cup L)\neq\emptyset	$. If $N(w)\cap L\neq\emptyset$, then $N(w)\cap L_1\neq\emptyset$, which implies $\partial(v,w)\leq \partial(v,L_1)+\partial(L_1,w)\leq 4$ by Claim \ref{cla6} and Orientation \ref{cons1}. If $N(w)\cap S_{3,3}\neq\emptyset$, then $N(w)\cap X\neq\emptyset$ by Proposition \ref{pro1}, and thus $\partial(v,w)\leq \partial(v,X)+\partial(X,w)\leq 4$ by Claim \ref{cla4} and Orientation \ref{cons1}. 

For $w\in K_2$, we have $N(w)\cap K_1\neq\emptyset$ and $K_1\rightarrow K_2$ by Orientation \ref{cons1}, and thus $\partial(v,w)\leq 5$ by the result of $w\in K_1$. 

For $w\in K_3$, we have $N(w)\cap M_1\neq\emptyset$ by Observation \ref{obs1}, and thus $\partial(v,w)\leq 5$ by Claim \ref{cla6} and Orientation \ref{cons1}. 

For $w\in I_2\cup J_2$, then $N(w)\cap X\neq\emptyset$ by Proposition \ref{pro1}, and thus $\partial(v,w)\leq 4$ if $w\in I_2$ and $\partial(w,u)\leq 4$ if $w\in J_2$ by Claim \ref{cla4} and Orientation \ref{cons1}. 

For $w\in I_4\cup I_5$, we have $N(w)\cap (I_2\cup K_1)\neq\emptyset$ if $w\in I_4$, and $N(w)\cap (K_2\cup K_3)\neq\emptyset$ if $w\in I_5$. Then $\partial(v,w)\leq 5$ if $w\in I_4$, and $\partial(v,w)\leq 6$ if $w\in I_5$ by Orientation \ref{cons1}. 

For $w\in J_4\cup J_5$, we have $N(w)\cap (J_2\cup L_1)\neq\emptyset$ if $w\in J_4$, and $N(w)\cap (L_2\cup L_3\cup L_4)\neq\emptyset$ if $w\in J_5$. Then $\partial(w,u)\leq 5$ if $w\in J_4$, and $\partial(w,u)\leq 8$ if $w\in J_5$ by Orientation \ref{cons1} and Claim \ref{cla6}.

For $w\in A_5\cup B_5$, we have $N(w)\cap I_2\neq\emptyset$ if $w\in A_5$, and $N(w)\cap J_2\neq\emptyset$ if $w\in B_5$, then $\partial(v,w)\leq 5$ if $w\in A_5$, and $\partial(w,u)\leq 5$ if $w\in B_5$ by Orientation \ref{cons1}. Similarly, the results hold for $w\in A_7\cup A_8\cup B_7\cup B_8$.
\end{proof}

\vspace{10pt}

To obtain the directed distances for $w\in B_{10}$, we define a partition of $B_{10}$ as follows.
\begin{align*}
B_{10}^{(1)}&=\{w\mid w\in B_{10}\text{ and }w \text{ is isolated in } G[B_{10}]\},\\ B_{10}^{(2)}&=\{w\mid w\in B_{10}\text{ and }w \text{ is not isolated in } G[B_{10}]\};
\end{align*}

\begin{Orientation}\label{cons2}
{\rm (i)}	Let $R=\{v\}$ and $S=B_{10}^{(2)}$. Orient the edges of $G[S]$ and $[R,S]$ by an $R$-$S$ orientation. \\ {\rm (ii)} Let $w\in B_{10}^{(1)}$ be an undirected vertex. Using the following iterative procedure, we assign orientations to some edges of $G$. If $w$ is not an isolated vertex in $G[B\cup B']$, then let $v\rightarrow N(w)\cap (B\cup B'-B_{10})\rightarrow w\rightarrow  v$.
If $w$ is isolated in $G[B\cup B']$, then $vw$ lies in a $2$-cycle, $4$-cycle or $5$-cycle since $g^*\in \{4,5\}$. From these cycles, select the shortest cycle and denote it as $C_1$. If $C_1$ is a $2$-cycle, then orient $C_1$ by $B$-$C$ orientation since $w$ is undirected. If $C_1$ is an $r$-cycle, where $r=4$ or $5$ $($and $r=5$ holds only if $g^*=5)$, then $C_1$ is given by either $vww_2w_3v$ or $vww_1w_2w_3v$. If $w_2\in B$, then $w$ is in a cycle of length $r-1$ since $w_2v\in E(G)$, a contradiction. By the definition of $B_{10}^{(1)}$, $w_1\in J'$, $w_2\in J'\cup J$ $(w_2\in J'$ if $r=4)$, $w_3\in B$. If $r=5$ and $w_1w_2$ is oriented, then $w_1w_2$ is in a directed $5$-cycle, and thus $vw$ is in a $4$-cycle, a contradiction. So, $w_1w_2$ is undirected when $r=5$. Since $w$ is undirected, $vww_2$ is undirected if $r=4$, and $vww_1w_2$ is undirected if $r=5$. 

If $C_1$ is mixed, then orient $C_1$ by a $B$-$C$ orientation. Otherwise, $w_2w_3$ is oriented, and $w_2\in J'$ by Orientation \ref{cons1}. Thus, there exists a directed cycle containing $v,w_2,w_3$, and hence there exists a mixed cycle $C_2$ derived from $C_1$ such that $vww_1\subseteq C_2$ and $|C_1|=|C_2|$. We then orient the edges of $C_2$ by a $B$-$C$ orientation. 
\end{Orientation}

\begin{claim}\label{cla7_1}
For $w\in B_{10}^{(1)}$, we have $\partial(w,u)\leq 2g^*-2$ and $\partial(v,w)\leq g^*-1$. For $w\in B_{10}^{(2)}$, we have $\partial(w,u)\leq g^*+1$ and $\partial(v,w)\leq 2$.
\end{claim}

\begin{proof}
If $w\in B_{10}^{(2)}$, by (i) of Orientation \ref{cons2} and Lemma \ref{l1}, we obtain $\theta(w,v)\leq 2$. Hence $\partial(w,u)\leq \partial(w,v)+\partial(v,u)\leq g^*+1$ and $\partial(v,w)\leq 2$ since $\partial(v,u)=g^*-1$ from Claim \ref{cla1}. By Orientation \ref{cons2}, if $w\in B_{10}^{(1)}$ and $w$ is not an isolated vertex in $G[B\cup B']$, then $\theta(v,w)\leq 2$; if $w\in B_{10}^{(1)}$ and $w$ is isolated in $G[B\cup B']$, then there exists a directed cycle of length at most $g^*$ containing $vw$, and thus $\theta(v,w)\leq g^*-1$. Since $\partial(v,u)= g^*-1$, we obtain the result as expected.
\end{proof}

The directed graph obtained from Orientation \ref{cons2} is denoted by $\rbjt{D_2}$. We define a partition for $B'$ as follows.
\begin{align*}
B'_1&=B'\cap V(\rbjt{D_2}),\  B'_2=\{w\mid w\in B'-B'_1\text{ and } w \text{ is isolated in } G[B'-B'_1]\},\\ B'_3&=\{w\mid w\in B'-B'_1\text{ and } w \text{ is not isolated in } G[B'-B'_1]\}.
\end{align*}

\begin{Orientation}\label{cons3}
Let $R=\{v\}$ and $S=B'_3$. Orient the edges of $G[S]$ and $[R,S]$ by an $R$-$S$ orientation. Let $w\in B'_2$ be an undirected vertex. If $w$ is not an isolated vertex in $G[B']$, then $N(w)\cap B'_1\neq\emptyset$, and thus let $B'_1\rightarrow w\rightarrow v$. Otherwise, $w$ is an isolated vertex in $G[B']$. If $|[w,v]|\geq 2$, then orient the edges in $[w,v]$ in two ways. If $|[w,v]|=1$, then $N(w)\cap B\neq \emptyset$ since $G$ is bridgeless. For some $w'\in N(w)\cap B$, $vww'v$ is a mixed 3-cycle since $w$ is undirected, and we then orient $vww'v$ by a $B$-$C$ orientation.
\end{Orientation}

\begin{claim}
For $w\in B'$, we have $\partial(w,u)\leq g^*+1$ and $\partial(v,w)\leq 2$.
\end{claim}

\begin{proof}
By Orientation \ref{cons2}, we have $v\rightarrow B'_1\rightarrow N(B'_1)\cap B_{10}^{(1)}\rightarrow v$. Then by Orientation \ref{cons3} and Lemma \ref{l1}, we have $\theta(v,w)\leq 2$ for any $w\in B'$. Hence the conclusion holds by $\partial(v,u)=g^*-1$ in Claim \ref{cla1}.
\end{proof}

By the symmetry of $u$ and $v$, we suppose that $|K'|\geq |L'|$, and then $K'\neq\emptyset$ if $L'\neq\emptyset$. To orient the edges incident to $L'$, we use the following notations.

For \(w \in L'\), let \(P_w = w i_1 \cdots i_\ell\) denote a shortest \((w, K')\)-path, where $\ell\in \{3,4\}$ and \(i_\ell \in K'\). If \(\ell = 3\), then \(i_1 \in J_1\cup L_1\) and $i_2\in I_1\cup K_1$. If $\ell=4$, then $i_1\in L'\cup L\cup (\cup_{i=1}^4 J_i)$, $i_2\in K_1\cup X\cup M_1\cup I_1\cup J_1\cup L_1\cup S_{2,2}$ and $i_3\in K\cup (\cup_{i=1}^4 I_i)$.
\begin{align*}
L'_1&=L'\cap N(L_1\cup L_2\cup L_3),\\
L'_2&=\{w\mid w\in L'-L'_1,\  \exists P_w=wi_1i_2i_3i_4 \text{ with }i_1\in L_4\cup L'_1,i_2\in X\cup L_1\cup J_1,\\  &\quad\  i_3\in I_1\cup I_2\cup K_1,i_4\in K'\},\\
L'_3&=\{w\mid w\in L'-L'_1-L'_2,\  w \text{ is isolated in } G[L'-L'_1-L'_2]\},\\
L'_4&=\{w\mid w\in L'-\cup_{i=1}^3L'_i,\ d(w,K')=3\},\  L'_5=\{w\mid w\in L'-\cup_{i=1}^3L'_i,\ d(w,K')=4\}.
\end{align*}

We further partition $L'_4,L'_5$ as follows.
\begin{align*}
L'_{11}&=L'_1\cap N(L_1\cup L_3),\ L'_{12}=L'_1\cap N(J_1)-L'_{11},\ L'_{13}=L'_1-L'_{11}-L'_{12};\\		
L'_{41}&=L'_4\cap N(L'_5),\ 
L'_{42}=\{w\mid w\in L'_{4}-L'_{41},\ w \text{ is isolated in } G[L'_{4}-L'_{41}]\},\\
L'_{43}&=\{w\mid w\in L'_{4}-L'_{41},\ w \text{ is not isolated in } G[L'_{4}-L'_{41}]\};\\
L'_{51}&=L'_5\cap N(L'_4),\\
L'_{52}&=\{w\in L'_{5}-L'_{51}\mid w \text{ is isolated in } G[L'_{5}-L'_{51}] \text{ or }N(w)\cap J_6\neq\emptyset\},\\
L'_{53}&=\{w\mid w\in L'_5-L'_{51}-L'_{52},\ w \text{ is isolated in } G[L'_{5}-L'_{51}-L'_{52}] \},\\
L'_{54}&=\{w\mid w\in L'_{5}-L'_{51}-L'_{52},\ w \text{ is not isolated in } G[L'_{5}-L'_{51}-L'_{52}]\}.
\end{align*}

For any $w\in L'_2$, we have $d(w,K')=4$. It follows that $L'\cap N(L_3)\subseteq L'_{11}$, so Orientation \ref{cons1} does not conflict with Orientation \ref{cons8_2}. We then obtain the following orientation.

\begin{Orientation}\label{cons8_2}
{\rm (i)} Let $J_5\cup J_6\rightarrow L'$, $J\rightarrow L'_{11}\cup L'_{13}\cup L'_2\rightarrow L$, $L'_2\rightarrow L'_{11}\cup L'_{12}$, $L_2\rightarrow L'_{12}\rightarrow J_1$,  $J\rightarrow L'_{51}\cup L'_{42}\rightarrow L'_{41}\rightarrow J_1$, $L'_{51}\rightarrow L'_{52}\rightarrow L'_{53}\rightarrow \cup_{i=1}^4 J_i$, $L'_{52}\rightarrow (\cup_{i=1}^4 J_i)\cup L'_{11}\cup L'_{12}$. 

{\noindent\rm (ii)} Let $w\in L'_3$. Then $P_w=wi_1i_2\cdots i_\ell$ with $i_1\in J_1$ if $\ell=3$, and $i_1\in \cup_{i=1}^4 J_i$ if $\ell =4$ by the definition of $L'_1$ and $L'_2$. If $w$ is not isolated in $G[L']$, then $N(w)\cap (L'_1\cup L'_2)\neq\emptyset$. Thus let $ L'_{11}\cup L'_{13}\cup L'_2\rightarrow w\rightarrow \cup_{i=1}^4 J_i$ when $N(w)\cap L'_{12}=\emptyset$, and $\cup_{i=1}^4 J_i\rightarrow w\rightarrow L'_{12}$ when $N(w)\cap L'_{12}\neq\emptyset$. Assume now $w$ is isolated in $G[L']$. When $N(w)\cap L\neq\emptyset$, we have $N(w)\cap L_4\neq\emptyset$, and let $L_4\rightarrow w\rightarrow i_1$. When $N(w)\cap L=\emptyset$, we have $|[w,J]|\geq2$ since $G$ is bridgeless. Orient $wi_1$ as $\rbjt{wi_1}$, and the other edges of $[w,J]$ toward $w$.

{\noindent\rm (iii)} Let $R=J_1,S=L'_{43}$, or $R=\cup_{i=1}^4 J_i,S=L'_{54}$. Then $N(w)\cap R\neq\emptyset$ for $w\in S$ by checking every shortest $(w,K')$-path. Thus we give an $R$-$S$ orientation to the edges of $G[S]$ and $[R,S]$.
\end{Orientation}

\begin{claim}\label{cla9_2}
For $w\in L'$, we have $\partial(v,w)\leq 4$, and $\partial(w,u)\leq 7$ with equality only if $w\in L'_{54}$.
\end{claim}

\begin{proof}
For $w\in L'_1\cup L'_2$, we have $N(w)\cap (L_1\cup L_3)\neq\emptyset$ if $w\in L'_{11}$, $N(w)\cap J_1\neq\emptyset$ if $w\in L'_{12}$,  and $N(w)\cap L_2\neq\emptyset$ if $w\in L'_{13}$. By Orientations \ref{cons1} and \ref{cons8_2}, we have $v\rightarrow (B-B_{10})\rightarrow J\rightarrow L'_{11}\cup L'_{13}\cup L'_2\rightarrow L$, $L'_2\rightarrow L'_{11}\cup L'_{12}$ and $L_2\rightarrow L'_{12}\rightarrow J_1$. Then $\partial(v,w)\leq 4$ with equality only if $w\in L'_{12}$ by Claim \ref{cla6}. If $w\in L'_{11}$, then $\partial(w,u)\leq \partial(w,L_1\cup L_3)+\partial(L_1\cup L_3,u)\leq 1+4=5$ by Claim \ref{cla6}. If $w\in L'_{12}$, then $\partial(w,u)\leq \partial(w,J_1)+\partial(J_1,u)\leq 4$. If $w\in L'_{13}$, then $\partial(w,u)\leq \partial(w,L_2)+\partial(L_2,u)\leq 6$. If $w\in L'_2$, there exists a shortest $(w,K')$-path $P_w=wi_1i_2i_3i_4$ with $i_1\in L_4\cup L'_1,i_2\in X\cup L_1\cup J_1,i_3\in I_1\cup I_2\cup K_1,i_4\in K'$. For $i_1\in L_4$, we have $i_2\in X,i_3\in I_1\cup I_2\cup K_1$ by the definition of $L_4$, and $w\rightarrow i_1\rightarrow i_2\rightarrow i_3$ by Orientation \ref{cons1}, and thus $\partial(w,u)\leq \partial(w,i_3)+\partial(i_3,u)\leq 3+3=6$ by Claims \ref{cla1} and \ref{cla7}. For $i_1\in L'_1$, we have $i_1\in L'_{11}\cup L'_{12}$ since $i_2\in L_1\cup J_1$, and thus $\partial(w,u)\leq \partial(w,i_1)+\partial(i_1,u)\leq 6$.

For $w\in L'_3$, $P_w=wi_1i_2\cdots i_\ell$ with $i_1\in J_1$ if $\ell=3$, and $i_1\in \cup_{i=1}^4 J_i$ if $\ell =4$. By Claims \ref{cla1}, \ref{cla7} and Orientation \ref{cons8_2}, if $w$ is not isolated in $G[L']$, then we have $\partial(v,w)\leq 4$ and $\partial(w,u)\leq \partial(w,w')+\partial(w',u)\leq 1+5=6$ for $w'\in N(w)\cap ((\cup_{i=1}^4 J_i)\cup L'_{12})$.  If $w$ is isolated in $G[L']$, then by Orientation \ref{cons8_2}, we have $w'\rightarrow w\rightarrow i_1$ for some $w'\in L_4\cup J$, which implies $\partial(v,w)\leq 3$ and $\partial(w,u)\leq 6$ by Orientation \ref{cons1}, Claims \ref{cla1}, \ref{cla6} and \ref{cla7}.

For $w\in L'_{41}\cup L'_{42}\cup L'_{51}$, we have $N(w)\cap J_1\neq \emptyset$ if $w\in L'_{41}\cup L'_{42}$; $N(w)\cap L'_{51}\neq\emptyset$ if $w\in L'_{41}$; $N(w)\cap L'_{41}\neq\emptyset$ if $w\in L'_{51}$ by the definitions of $L'_{41}$, $L'_{42}$, $L'_{51}$ and checking every shortest $(w,K')$-path. If $w\in L'_{42}$, then $w$ is not isolated in $G[L'_4\cup L'_5]$ by the definition of $L'_3$, and thus $N(w)\cap (L'_4\cup L'_5)\neq\emptyset$, which implies $N(w)\cap L'_{41}\neq\emptyset$ by the definition of $L'_{41}$. By Orientations \ref{cons1} and \ref{cons8_2}, we have $v\rightarrow (B-B_{10})\rightarrow J\rightarrow L'_{51}\cup L'_{42}\rightarrow L'_{41}\rightarrow J_1\rightarrow I_1\rightarrow A_1\rightarrow u$. Thus $\partial(v,w)\leq 4$ and $\partial(w,u)\leq 5$ for any $w\in L'_{41}\cup L'_{42}\cup L'_{51}$.

For $w\in L'_{52}\cup L'_{53}$, there exists a shortest $(w,K')$-path $P_w=wi_1i_2i_3i_4$ such that $i_1\in \cup_{i=1}^4 J_i$ by the definition of $L'_{51}$. Then $w\rightarrow \cup_{i=1}^4 J_i$ by Orientation \ref{cons8_2}, and thus $\partial(w,u)\leq 6$ by Claims \ref{cla1} and \ref{cla7}. By Orientations \ref{cons1} and \ref{cons8_2}, if $w\in L'_{52}$ and $w$ is isolated in $G[L'_5-L'_{51}]$, then $N(w)\cap L'_{51}\neq\emptyset$, and thus we have $ v\rightarrow(B-B_{10})\rightarrow J\rightarrow L'_{51}\rightarrow w$, which implies $\partial(v,w)\leq 4$; if $w\in L'_{52}$ and $N(w)\cap J_6\neq\emptyset$, then we have $J_6\rightarrow w$, which implies $\partial(v,w)\leq 3$; if $w\in L'_{53}$, then $w$ has a neighbor $w'\in L'_{52}$ and $N(w')\cap J_6\neq\emptyset$, and thus $w'\rightarrow w$, which implies $\partial(v,w)\leq 4$. Hence $\partial(v,w)\leq 4$ for any $w\in L'_{52}\cup L'_{53}$.

For $w\in L'_{43}\cup L'_{54}$, by the definitions of $L'_{43}$, $L'_{54}$, Orientations \ref{cons1}, \ref{cons8_2} and Lemma \ref{l1}, if $w\in L'_{43}$, then $v\rightarrow (B-B_{10})\rightarrow J_1\rightarrow I_1\rightarrow A\rightarrow u$ and $\theta(w,J_1)\leq 2$, and thus $\partial(v,w)\leq 4$ and $\partial(w,u)\leq 5$; if $w\in L'_{54}$, then $\partial(v,w')\leq 2$, $\partial(w',u)\leq 5$ for any $w'\in \cup_{i=1}^4 J_i$ by Claims \ref{cla1} and \ref{cla7}, and $\theta(w,\cup_{i=1}^4 J_i)\leq 2$, and hence $\partial(v,w)\leq 4$ and $\partial(w,u)\leq 7$.
\end{proof}
\vspace{10pt}

\begin{claim}\label{cla10}
{\rm (i)} For $w\in J_6$, we have $\partial(v,w)=2$ and $\partial(w,u)\leq 7$.\\
{\rm (ii)} For $w\in B_9$, we have $\partial(v,w)=1$ and $\partial(w,u)\leq 8$.	\end{claim}

\begin{proof}
By Orientation \ref{cons1}, we have $\partial(v,w)=2$ if $w\in J_6$, and $\partial(v,w)=1$ if $w\in B_9$.  For $w\in J_6$, we have $N(w)\cap (L'-L'_{54})\neq\emptyset$ by the definitions of $L'_{52}$, which implies $\partial(w,u)\leq \partial(w,w')+\partial(w',u)\leq 1+6=7$ for $w'\in N(w)\cap (L'-L'_{54})$ by Orientation \ref{cons8_2} and Claim \ref{cla9_2}. For $w\in B_9$, we have $N(w)\cap J_6\neq\emptyset$, which implies $\partial(w,u)\leq \partial(w,J_6)+\partial(J_6,u)\leq 1+7=8$ by Orientation \ref{cons1}.
\end{proof}
\vspace{10pt}

Recall that $\rbjt{D_2}$ is the directed graph obtained from Orientation \ref{cons2}. To orient the edges of $G[J']$, $[B,J']$ and $[J,J']$, we partition $J'$ as follows. 
\begin{align*}
J'_1&=J'\cap V(\rbjt{D_2}),\quad J'_2=(J'-J'_1)\cap N(\cup_{i=1}^3 J_i), \quad J'_3=(J'-J'_1-J'_2)\cap N(\cup_{i=4}^6 J_i),\\
J'_4&=\{w\mid w\in J'- \cup_{i=1}^3 J'_i,\ w \text{ is isolated in } G[J'- \cup_{i=1}^3 J'_i]\},\\
J'_5&=(J'-\cup_{i=1}^4 J'_i)\cap N(\cup_{i=1}^6 B_i),\quad J'_6=(J'-\cup_{i=1}^5 J'_i)\cap N(\cup_{i=7}^{10} B_i).
\end{align*}
We further partition $J'_3$, $J'_4$, $J'_5$ and $J'_6$ as follows.
\begin{align*}
J'_{31}&=J'_3\cap N(\cup_{i=1}^6 B_i),\  J'_{32}=J'_3-J'_{31};\
J'_{41}=J'_4\cap N(\cup_{i=1}^6 B_i),\  J'_{42}=J'_4-J'_{41};\\		
J'_{51}&=J'_5\cap N(J'_6),\  J'_{52}=\{w\mid w\in J'_5-J'_{51},\ w\text{ is isolated in } G[J'_5-J'_{51}]\},\\ J'_{53}&=\{w\mid w\in J'_5-J'_{51},\ w\text{ is not isolated in } G[J'_5-J'_{51}]\};\\
J'_{61}&=J'_6\cap N(J'_5),\  J'_{62}=\{w\mid w\in J'_6-J'_{61},\ w\text{ is isolated in } G[J'_6-J'_{61}]\},\\ J'_{63}&=\{w\mid w\in J'_6-J'_{61},\ w\text{ is not isolated in } G[J'_6-J'_{61}]\}.
\end{align*}	

\begin{Orientation}\label{cons81}
{\rm (i)} Let $B\rightarrow J'_2\rightarrow J$, $ B_8\cup B_9\rightarrow J'_1$, $J'_{62}\rightarrow \cup_{i=7}^{10} B_i\rightarrow J'_{61}\rightarrow J'_{51}\cup J'_{62}$, $\cup_{i=1}^6 B_i\rightarrow J'_{52}\rightarrow J'_{51}\rightarrow \cup_{i=1}^6 B_i$, $B\rightarrow J'_{32}\rightarrow\cup_{i=4}^6 J_i\rightarrow J'_{31}\rightarrow \cup_{i=1}^6 B_i$, $J'_{32}\cup J'_{42}\cup J'_{6}\rightarrow J'_1\cup J'_2\cup J'_{31}\cup J'_{41}\cup J'_{5}$. 
Let $R=\cup_{i=1}^6 B_i$, $S=J'_{53}$, or $R=\cup_{i=7}^{10} B_i$, $S=J'_{63}$. Then $G[S]$ has no isolated vertex, and thus orient the edges of $G[S]$ and $[R,S]$ by an $R$-$S$ orientation.


{\noindent\rm (ii)} For $w\in J'_4$, $N(w)\cap B\neq\emptyset$. If $w$ is not an isolated vertex of $G[J']$, then $N(w)\cap (\cup_{i=1}^3 J'_i)\neq\emptyset$, and thus let $\cup_{i=1}^3 J'_i\rightarrow w\rightarrow \cup_{i=1}^6 B_i$ if $w\in J'_{41}$, and $\cup_{i=7}^{10} B_i\rightarrow w\rightarrow\cup_{i=1}^3 J'_i$ if $w\in J'_{42}$. Otherwise, $w$ is an isolated vertex of $G[J']$. Since $N(w)\cap J=\emptyset$, we have $|[w,B]|\geq 2$. If $w\in J'_{41}$, then $w\rightarrow w'$ for some $w'\in N(w)\cap (\cup_{i=1}^6 B_i)$, and the other edges of $[w,B]$ toward $w$; if $w\in J'_{42}$, then we orient the edges of $[w,B]$ in two ways.		
\end{Orientation}

\begin{claim}\label{cla9_1}
Table \ref{TableJprime} presents upper bounds on directed distances for $J'$.
\begin{table}[H]
\centering

\renewcommand{\arraystretch}{1.1}
\setlength{\tabcolsep}{4pt}
\begin{tabular}{|l|*{9}{>{\centering\arraybackslash}p{1.2cm}|}}
	\hline
	\textbf{for \(w\) in} & $J'_1$ & $J'_2$ & $J'_{31}$ & $J'_{32}$ & $J'_{41}$ & $J'_{42}$ & $J'_5$ & $J'_6$ \\
	\hline
	$\partial(w,u) \le$ & $2g^*-3$ & $5$ & $6$ & $9$ & $6$ & $10$ & $7$ & $11$ \\
	\hline
	$\partial(v,w) \le$ & $g^*-2$ & $g^*$ & $3$ & $g^*$ & $g^*+1$ & $g^*$ & $g^*+1$ & $g^*+1$ \\
	\hline
\end{tabular}\vspace{-7pt}
\caption{Upper bounds on directed distances for $w\in J'$.}
\label{TableJprime}
\vspace{-12pt}
\end{table}
\end{claim}

\begin{proof}
For $w\in J'_1$, by Orientation \ref{cons2}, there exists a directed cycle of length at most $g^*$ containing $v$ and $w$. Then $\theta(v,w)\leq g^*-2$, and thus $\partial(w,u)\leq 2g^*-3$ and $\partial(v,w)\leq g^*-2$ since $\partial(v,u)=g^*-1$ in Claim \ref{cla1}.

By Claims \ref{cla1}, \ref{cla7}, \ref{cla7_1} and \ref{cla10}, we have $\partial(w,u)\leq 4$ and $\partial(v,w)\leq 2$ for $w\in \cup_{i=1}^3 J_i$, $\partial(w,u)\leq 8$ and $\partial(v,w)\leq 2$ for $w\in \cup_{i=4}^6 J_i$, $\partial(w,u)\leq 5$ and $\partial(v,w)\leq 1$ for $w\in \cup_{i=1}^6 B_i$, and $\partial(w,u)\leq 9$ and $\partial(v,w)\leq g^*-1$ for $w\in \cup_{i=7}^{10} B_{i}$. We next analyze the directed distance of $J'-J'_1$.

For $w\in J'_2$, we have $N(w)\cap (\cup_{i=1}^3 J_i)\neq\emptyset$. Then by Orientation \ref{cons81}, we have $B\rightarrow J'_2\rightarrow \cup_{i=1}^3 J_i$, which implies $\partial(v,w)\leq g^*$ and $\partial(w,u)\leq 5$. 

For $w\in J'_3$, we have $N(w)\cap (\cup_{i=4}^6 J_i)\neq\emptyset$. Then by Orientation \ref{cons81}, we have $\partial(w,u)\leq 6$ and $\partial(v,w)\leq 3$ if $w\in J'_{31}$, and $\partial(w,u)\leq 9$ and $\partial(v,w)\leq g^*$ if $w\in J'_{32}$.

For $w\in J'_4$, by (ii) of Orientation \ref{cons81},  we have $\partial(w,u)\leq 6$ and $\partial(v,w)\leq g^*+1$ if $w\in J'_{41}$, $\partial(w,u)\leq 10$ and $\partial(v,w)\leq g^*$ if $w\in J'_{42}$.

For any $w\in J'_{51}\cup J'_{52}\cup J'_{61}\cup J'_{62}$, the following hold by definition.
If $w\in J'_{61}$, then $N(w)\cap J'_{51}\neq\emptyset$; if $w\in J'_{51}$, then $N(w)\cap J'_{61}\neq\emptyset$.
By the definition of $J'_4$, every vertex $w\in J'_{52}$ satisfies $N(w)\cap (J'_5\cup J'_6)\neq\emptyset$, which yields $N(w)\cap J'_{51}\neq\emptyset$.
Similarly, we have $N(w)\cap J'_{61}\neq\emptyset$ for each $w\in J'_{62}$. By Orientation \ref{cons81}, we have $J'_{62}\rightarrow\cup_{i=7}^{10} B_i\rightarrow J'_{61}\rightarrow J'_{51}\cup J'_{62}$, and $\cup_{i=1}^6 B_i\rightarrow J'_{52}\rightarrow J'_{51}\rightarrow \cup_{i=1}^6 B_i$. Then $\partial(w,u)\leq 6$ and $\partial(v,w)\leq g^*+1$ if $w\in J'_{51}$, $\partial(w,u)\leq 7$ and $\partial(v,w)\leq 2$ if $w\in J'_{52}$, $\partial(w,u)\leq 7$ and $\partial(v,w)\leq g^*$ if $w\in J'_{61}$, $\partial(w,u)\leq 10$ and $\partial(v,w)\leq g^*+1$ if $w\in J'_{62}$.

For $w\in J'_{53}\cup J'_{63}$, by Lemma \ref{l1}, we have $\theta(w,\cup_{i=1}^6 B_i)\leq 2$ if $w\in J'_{53}$, and $\theta(w,\cup_{i=7}^{10} B_i)\leq 2$ if $w\in J'_{63}$, which implies $\partial(w,u)\leq 7$ and $\partial(v,w)\leq 3$ if $w\in J'_{53}$, and $\partial(w,u)\leq 11$ and $\partial(v,w)\leq g^*+1$ if $w\in J'_{63}$.
\end{proof}

To orient some edges incident to $I_6$,	we define a partition of $I_6$ as follows.
$$
I_{61}=I_6\cap N(\cup_{i=1}^3 A_i),\ I_{62}=I_6-I_{61}.
$$
We further partition $I_{61}$ as follows.
\begin{align*}
I_{61}^{(1)}&=I_{61}\cap N(I_{62}),\\
I_{61}^{(2)}&=\{w\mid w\in I_{61}-I_{61}^{(1)}\text{, }w \text{ is isolated in } G[I_{61}-I_{61}^{(1)}]\},\\
I_{61}^{(3)}&=\{w\mid w\in I_{61}-I_{61}^{(1)}\text{, }w \text{ is not isolated in } G[I_{61}-I_{61}^{(1)}]\};
\end{align*}

\begin{Orientation}\label{cons8}
{\noindent\rm (i)} Let $\cup_{i=1}^3 A_i\rightarrow I_{61}^{(1)}\rightarrow I_{62}$, $R=\cup_{i=1}^3 A_i$ and $S=I_{61}^{(3)}$. We then orient the edges of $G[S]$ and $[R,S]$ by an $R$-$S$ orientation. 

{\noindent\rm (ii)} Let  $w\in I_{61}^{(2)}$ be an undirected vertex. Then $N(w) \cap (\cup_{i=1}^3 A_i) \neq \emptyset$. Using the following iterative procedure, we assign orientations to some edges of $G$ such that  $I\rightarrow \cup_{i=4}^9 A_i$. If $w$ is not isolated in $G[I\cup I']$, then $N(w)\cap ((I-I_6)\cup I'\cup I_{61}^{(1)})\neq \emptyset$, and thus let $\cup_{i=1}^3 A_i\rightarrow w\rightarrow (I-I_6)\cup I'\cup I_{61}^{(1)}$ and $N(w)\cap ((I-I_6)\cup I')\rightarrow A$. Now we assume $w$ is isolated in $G[I\cup I']$.  If $|[w,A]|\geq 2$, then orient $w'w$ as $w'\rightarrow w$ for some $w'\in N(w)\cap (\cup_{i=1}^3 A_i)$, and the other edges of $[w,A]$ away from $w$. If $|[w,A]|=1$, then $N(w)\cap K'\neq\emptyset$. Let $w_0 \in N(w) \cap(\cup_{i=1}^3 A_i)$ and $w_1\in N(w)\cap K'$ such that $|N(w_1)\cap I|$ is maximum.

For $|N(w_1)\cap I|\geq 2$, let $w_2\in N(w_1)\cap I-\{w\}$ and $w_3\in N(w_2)\cap A$ such that $w_1w_2w_3$ is a mixed path by iteration. Moreover, if $w_3\rightarrow w_2$, then $w_3\in \cup_{i=1}^3 A_i$. Since $w$ is undirected, $w_0ww_1w_2w_3$ is a mixed path or cycle. If $w_0ww_1w_2w_3$ is undirected and $w_0\neq w_3$, then orient $w_0ww_1w_2w_3$ as a directed path from $w_0$ to $w_3$. Otherwise, we orient  $w_0ww_1w_2w_3$ by a $B$-$C$ or $B$-$P$ orientation. 

For $|N(w_1)\cap I|= 1$, we have $|N(x)\cap I|=1$ for any $x\in N(w)\cap K'$ by the maximality of $|N(w_1)\cap I|$. Then since $g^*\in\{4,5\}$ and $|[w,A]|=1$, the shortest cycle containing $w_0w$ is a $5$-cycle $C=w_0wy_1y_2y_3w_0$, where $y_1\in K'$, $y_2\in K'\cup K$ and $y_3\in I$, which implies $g^*=5$.  Since $w$ is undirected and $|N(y_1)\cap I|=1$, we see $w_0wy_1y_2$ is undirected. If $C$ is mixed, then orient the edges of $C$ by a $B$-$C$ orientation. If $C$ is not mixed, there exists a mixed path (or cycle) $C'=w_0wy_1y_2y_3y'_4$ such that $y'_4\in A$, and $y'_4 \in \cup_{i=1}^3 A_i$ if $y_3\rightarrow y_2$. We then orient the edges of $C'$ by a $B$-$P$ or $B$-$C$ orientation.
\end{Orientation}

The directed graph obtained by Orientation \ref{cons8} is denoted by $\rbjt{D_3}$. To orient the edges incident to $A_9$, we give some notations as follows:
\begin{align*}
I_{61}^{(21)}&=\{w\mid w\in I_{61}^{(2)}\text{ and } w'\rightarrow w \text{ in $\rbjt{D_3}$ for some }w'\in N(w)\cap A\},\\
I_{61}^{(22)}&=\{w\mid w\in I_{61}^{(2)}\text{ and } w\rightarrow w' \text{ in $\rbjt{D_3}$ for any } w'\in N(w)\cap A\};\\
I'_1&=I'\cap V(\rbjt{D_3}),\quad I'_2=(I'-I'_1)\cap N(I_{61}-I_{61}^{(22)});\\ 
I_{62}^{(1)}&=I_{62}\cap N(I_{61}),\quad I_{62}^{(2)}=(I_{62}-I_{62}^{(1)})\cap V(\rbjt{D_3}),\quad I_{62}^{(3)}=I_{62}-I_{62}^{(1)}-I_{62}^{(2)}.
\end{align*}		

It follows that $I_{61}^{(21)},I_{61}^{(22)}$ form a partition of $I_{61}^{(2)}$, $I'_1, I'_2 \subseteq I'$ with $I'_1 \cap I'_2 = \emptyset$, and $I_{62}^{(1)}, I_{62}^{(2)}, I_{62}^{(3)}$ form a partition of $I_{62}$. Furthermore, we define a partition of $A_9$ as follows.
\begin{align*}
A_{91}&=A_9\cap N(I'_1),\quad A_{92}=(A_9-A_{91})\cap N(I'_2),\\
A_{93}&=(A_9-A_{91}-A_{92})\cap N(I_{61}\cup I_{62}^{(1)}\cup I_{62}^{(2)}),\\
A_{94}&=\{w\mid w\in A_9-\cup_{i=1}^3 A_{9i}\text{, } w\text{ is isolated in } G[A_9-\cup_{i=1}^3 A_{9i}]\},\\
A_{95}&=\{w\mid w\in A_9-\cup_{i=1}^3 A_{9i}\text{, } w\text{ is not isolated in } G[A_9-\cup_{i=1}^3 A_{9i}]\};		
\end{align*}

From the above arguments, $I_6=I_{61}\cup I_{62}=I_{61}^{(1)}\cup I_{61}^{(21)}\cup I_{61}^{(22)}\cup I_{61}^{(3)}\cup I_{62}^{(1)}\cup I_{62}^{(2)}\cup I_{62}^{(3)}$.

\begin{Orientation}\label{cons8.5}
{\noindent\rm (i)} Let $\cup_{i=1}^3 A_i\rightarrow I_{61}^{(1)}\rightarrow I_{62}^{(1)}\cup I_{62}^{(2)}\rightarrow (A-A_{94}-A_{95})\rightarrow u$, $(I_{61}-I_{61}^{(22)})\rightarrow I'_2$, $(I-I_{62})\rightarrow I'_1\cup I'_2\rightarrow (A-A_{94}-A_{95})$, $I_{61}\rightarrow A_{93}$, $R=\{u\}$, $S=A_{95}$. Then orient the edges of $[R,S]$ and $G[S]$ by an $R$-$S$ orientation. 

{\noindent\rm (ii)} Let $w\in A_{94}$ be an undirected vertex. Using the following iterative procedure, we assign orientations to some edges of $G$. If $w$ is not an isolated vertex in $G[A\cup A']$, then $N(w)\cap (A\cup A'-A_{94}-A_{95})\neq\emptyset$, and thus let $u\rightarrow w\rightarrow N(w)\cap (A\cup A'-A_{94}-A_{95}) \rightarrow u$.

Suppose $w$ is an isolated vertex in $G[A\cup A']$. Then by $g^*\in \{4,5\}$, we have $uw$ in a $2$-cycle, $4$-cycle or $5$-cycle. From these cycles, select the shortest cycle and denote it as $C_1$. If $C_1$ is a $2$-cycle, then orient the edges of $C_1$ by $B$-$C$ orientation. If $C_1$ is an $r$-cycle, where $r=4$ or $5$ $($and $r=5$ holds only if $g^*=5)$, then $C_1$ is given by either $uww_1w_2u$ or $uww_1w^*w_2u$. If $w^*\in A$, then $w$ is in a $4$-cycle since $w^*u\in E(G)$, a contradiction to $r=5$. By the definitions of $A_{94}$ and $I'_2$, we have $w_1\in (I'-I'_1- I'_2)\cup I_{62}^{(3)}$,  $w^*\in (I'\cup I-I_{61})\cup I_{61}^{(22)}$ and $w_2\in A$. Since $w$ is undirected and the argument is similar to Orientation \ref{cons2}, $uww_1$ is undirected if $r=4$, and $uww_1w^*$ is undirected if $r=5$. 

If $C_1$ is mixed, then orient the edges of $C_1$ by a $B$-$C$ orientation. If $C_1$ is not mixed, then neither $w_1w_2u$ for $r=4$ nor $w^*w_2u$ for $r=5$ is mixed. 
We claim $w_1\notin V(\rbjt{D_3})$ if $r=4$, and $w^*\notin V(\rbjt{D_3})\cup (\cup_{i=1}^5 I_i)\cup I_{62}^{(1)}\cup I_{62}^{(2)}$ if $r=5$. Since $w_1\in (I'-I'_1- I'_2)\cup I_{62}^{(3)}$, we have $w_1\notin V(\rbjt{D_3})$. If $w^*\in V(\rbjt{D_3})\cup (\cup_{i=1}^5 I_i)\cup I_{62}^{(1)}\cup I_{62}^{(2)}$, then $w^*\in I'_{1}\cup (\cup_{i=1}^5 I_i)\cup I_{61}^{(22)}\cup  I_{62}^{(1)}\cup I_{62}^{(2)}$, which implies $w_2\in A-A_{94}-A_{95}$. By Orientation \ref{cons1} and the definition of $I_{61}^{(22)}$, we obtain $w^*\rightarrow w_2\rightarrow u$, yielding a contradiction. Then $w_1,w^*\notin V(\rbjt{D_3})$ and $w^*\in (I'-I'_1)\cup I_{62}^{(3)}$, which yields $w_1w_2$ for $r=4$ and $w^*w_2$ for $r=5$ remain undirected in Orientation \ref{cons8}. By iteration, there exists a directed cycle $C_2$ containing $w_1,w_2,u$ when $r=4$ and $w^*,w_2,u$ when $r=5$. Hence there exists a mixed cycle $C_3$ derived from $C_1$ and $C_2$ such that $uw\in E(C_3)$ and $|C_1|=|C_2|$. We then orient the edges of $C_3$ by a $B$-$C$ orientation.
\end{Orientation}

The directed graph obtained by Orientation \ref{cons8.5} is denoted by $\rbjt{D_4}$.

\begin{claim}\label{cla9}
{\rm (i)} For $w\in I_{61}^{(1)}$, $\partial(w,u)\leq 3$ and $\partial(v,w)\leq 4$. \\{\rm (ii)} For $w\in I_{61}^{(21)}$, $\partial(w,u)\leq g^*$ and $\partial(v,w)\leq 4$.\\{\rm (iii)} For $w\in I_{61}^{(22)}$, $\partial(w,u)\leq 2$ and $\partial(v,w)\leq g^*+2$.\\ 
{\rm (iv)} For $w\in I_{61}^{(3)}$, $\partial(w,u)\leq 3$ and $\partial(v,w)\leq 5$.\\ {\rm (v)} For $w\in I_{62}^{(1)}$, $\partial(w,u)\leq 2$ and $\partial(v,w)\leq 5$. \\ {\rm (vi)} For $w\in I_{62}^{(2)}$, $\partial(w,u)\leq 2$ and $\partial(v,w)\leq g^*+2$.
\end{claim}

\begin{proof}
If $w\in I_{61}\cup I_{62}^{(1)}\cup I_{62}^{(2)}$, then $N(w)\cap (A-A_{94}-A_{95})\neq\emptyset$ by the definition of $A_{93}$.

For $w\in I_{61}^{(1)}\cup I_{62}^{(1)}$, we have $N(w)\cap I_{62}^{(1)}\neq\emptyset$ if $w\in I_{61}^{(1)}$, and $N(w)\cap I_{61}^{(1)}\neq\emptyset$ if $w\in I_{62}^{(1)}$, and $\cup_{i=1}^3 A_i\rightarrow I_{61}^{(1)}\rightarrow I_{62}^{(1)}\rightarrow  (A-A_{94}-A_{95})\rightarrow u$ by Orientation \ref{cons8.5}, which implies (i) and (v) hold by Claim \ref{cla1}.

For $w\in I_{61}^{(3)}$, we have $\theta(w,\cup_{i=1}^3A_i)\leq 2$ by Orientation \ref{cons8} and Lemma \ref{l1}, and then (iv) holds by Claim \ref{cla1}. 

For $w\in I_{62}^{(2)}$, by the definition of $I_{62}^{(2)}$ and Orientation \ref{cons8}, there is a directed path or cycle $w_0w_1\dots w_{t-1} w_{t}$ such that $t\leq g^*$, $w_0\in \cup_{i=1}^3 A_i$ , $w_{t-1}=w$ and $w_{t}\in A-A_{94}-A_{95}$, then $w_0\rightarrow w_1\rightarrow \dots \rightarrow w_{t-1}\rightarrow w_t\rightarrow u$ by Orientation \ref{cons8.5}. Then $\partial(w,u)\leq 2$ and $\partial(v,w)\leq g^*+2$ by Claim \ref{cla1}, which implies (vi) holds. 

For $w\in I_{61}^{(2)}$, by Orientations \ref{cons8}, \ref{cons8.5}, the definition of $I_{61}^{(21)}$, if $w$ is not an isolated vertex in $G[I\cup I']$, then $\cup_{i=1}^3 A_i\rightarrow w\rightarrow (I-I_6)\cup I'_1\rightarrow (A-A_9)\cup A_{91}\rightarrow u$ or $\cup_{i=1}^3 A_i\rightarrow w\rightarrow I_{61}^{(1)}$, which implies $w\in I_{61}^{(21)}$, $\partial(v,w)\leq 4$ and $\partial(w,u)\leq 4$ by Claim \ref{cla1} and (i); if $w$ is an isolated vertex in $G[I\cup I']$ and $|[w,A]|\geq 2$, then there exist $w'\in \cup_{i=1}^3 A_i, w''\in A-A_{94}-A_{95}$ such that $w'\rightarrow w\rightarrow w''\rightarrow u$, which implies $w\in I_{61}^{(21)}$, and $\partial(v,w)\leq 4$ and $\partial(w,u)\leq 2$. 

Assume now $w$ is an isolated vertex in $G[I\cup I']$ and $|[w,A]|=1$. By Orientation \ref{cons8}, if $|N(w_1)\cap I|\geq 2$, then there exists a directed path (or cycle) $w_0wx_1x_2x_3$ if $w\in I_{61}^{(21)}$, and $w_0x_1x_2wx_3$ if $w\in I_{61}^{(22)}$ such that $w_0\in \cup_{i=1}^3 A_i$ and $x_3 \in A-A_{94}-A_{95}$; 
if $|N(w_1)\cap I|= 1$, then $g^*=5$ and there exists a directed path (or cycle) $w_0wz_1z_2z_3z_4$ if $w\in I_{61}^{(21)}$, and $w_0z_1z_2z_3wz_4$ if $w\in I_{61}^{(22)}$ such that $w_0\in \cup_{i=1}^3 A_i$ and $z_4\in A-A_{94}-A_{95}$. Therefore, (ii) and (iii) follow from Claim \ref{cla1} and $(A-A_{94}-A_{95})\rightarrow u$ in Orientation \ref{cons8.5}.
\end{proof}

\begin{claim}\label{cla9.5}
Table \ref{Table4} presents upper bounds on directed distances for $I_1'\cup I_2'\cup A_9$.
\begin{table}[ht]
\centering

\renewcommand{\arraystretch}{1.1}
\setlength{\tabcolsep}{4pt}
\begin{tabular}{|l|*{8}{>{\centering\arraybackslash}p{1.2cm}|}}
	\hline
	\textbf{for \(w\) in} &$I'_1$&$I'_2$&$A_{91}$&$A_{92}$&$A_{93}$&$A_{94}$&$A_{95}$\\
	\hline
	$\partial(w,u)\le$ &$2$&$2$&$1$&$1$&$1$&$g^*-1$&$2$\\
	\hline
	$\partial(v,w)\le$ &$5$&$6$&$6$&$7$&$g^*+3$&$2g^*-2$&$g^*+1$\\
	\hline
\end{tabular}
\vspace{-5pt}
\caption{Upper bounds on directed distances for some $w\in I'_1\cup I'_2\cup A_9$.}
\label{Table4}
\vspace{-12pt}
\end{table}
\end{claim}
\begin{proof}
For $w\in I'_1$, there exists $w'\in N(w)\cap I_{61}^{(2)}$ such that $\cup_{i=1}^3 A_i\rightarrow w'\rightarrow w\rightarrow (A-A_{94}-A_{95})\rightarrow u$ by Orientations \ref{cons8} and \ref{cons8.5}, which implies the result holds by Claim \ref{cla1} and $N(w)\cap \left(A-A_{94}-A_{95}\right)\neq\emptyset$. 

For $w\in I'_2$, we have $(I_{61}-I_{61}^{(22)})\rightarrow w\rightarrow (A-A_{94}-A_{95})\rightarrow u$ by Orientation \ref{cons8.5}, which implies that the desired bounds hold by Claim \ref{cla9} and the fact $I_{61}-I_{61}^{(22)}=I_{61}^{(1)}\cup I_{61}^{(21)}\cup I_{61}^{(3)}$.

For $w\in A_{91}\cup A_{92}$, the result holds by the results of $w\in I'_1\cup I'_2$ and Orientation \ref{cons8.5}.

For $w\in A_{93}$, we have $I_{61}\cup I_{62}^{(1)}\cup I_{62}^{(2)}\rightarrow A_{93}\rightarrow u$ by Orientation \ref{cons8.5}, which implies $\partial(w,u)=1$ and $\partial(v,w)\leq g^*+3$ by Claim \ref{cla9}. 

For $w\in A_{95}$, by Lemma \ref{l1} and Orientation \ref{cons8.5}, we have $\theta(w,u)\leq 2$. Then $\partial(v,w)\leq \partial(v,u)+\partial(u,w)\leq g^*-1+2= g^*+1$ since $\partial(v,u)= g^*-1$ in Claim \ref{cla1}.		

For $w\in A_{94}$, by Orientation \ref{cons8.5}, if $w$ is not an isolated vertex in $G[A\cup A']$, then $\theta(u,w)\leq 2$; if $w$ is isolated in $G[A\cup A']$, then $uw$ is in a directed cycle of length at most $g^*$, and thus $\theta(u,w)\leq g^*-1$. Since $\partial(v,u)= g^*-1$, the desired bounds hold.		
\end{proof}

\vspace{0.3\baselineskip}
To orient some edges incident to $A'$, we define a partition for $A'$ as follows.
\begin{align*}
A'_1&=A'\cap V(\rbjt{D_4}),\  A'_2=\{w\mid w\in A'-A'_1\text{ and } w \text{ is isolated in } G[A'-A'_1]\},\\ A'_3&=\{w\mid w\in A'-A'_1\text{ and } w \text{ is not isolated in } G[A'-A'_1]\}.
\end{align*} 
\begin{Orientation}\label{cons10}
Let $R=\{u\}$ and $S=A'_3$. Orient the edges of $G[S]$ and $[R,S]$ by an $R$-$S$ orientation. Let an undirected vertex $w\in A'_2$. If $w$ is not an isolated vertex in $G[A']$, then $N(w)\cap A'_1\neq\emptyset$, and thus let $u\rightarrow w\rightarrow A'_1\rightarrow u$. Otherwise, $w$ is an isolated vertex in $G[A']$. If $|[w,u]|\geq 2$, then orient the edges in $[w,u]$ in two ways. If $|[w,u]|=1$, then $N(w)\cap A\neq \emptyset$ since $G$ is bridgeless. For some $w'\in N(w)\cap A$, we know that $uww'u$ is a mixed 3-cycle since $w$ is undirected, and then orient $uww'u$ by a $B$-$C$ orientation.

\end{Orientation}

\begin{claim}\label{cla11}
For $w\in A'$, we have $\partial(w,u)\leq 2$ and $\partial(v,w)\leq g^*+1$. 
\end{claim}

\begin{proof}
By Orientations \ref{cons8.5}, \ref{cons10} and Lemma \ref{l1}, we have $\theta(u,w)\leq 2$ for any $w\in A'$. Since $\partial(v,u)=g^*-1$ in Claim \ref{cla1}, we obtain the result. 
\end{proof}

Recall that $I'_1=I'\cap V(\rbjt{D_3})$ and $I'_2=(I'-I'_1)\cap N(I_{61}-I_{61}^{(22)})$. To orient the edges of $G[I']$ and $[I',I\cup A]$, we partition $I'-(I'_1\cup I'_2)$ as follows.
\begin{align*}
I'_3&=(I'-I'_1-I'_2)\cap V(\rbjt{D_4}),\quad
I'_4=(I'-\cup_{i=1}^{3}I'_i)\cap N(I-I_{62}^{(3)}),\\ I'_5&=(I'-\cup_{i=1}^{4}I'_i)\cap N(I_{62}^{(3)}),\quad
I'_6=\{w\mid w\in I'-\cup_{i=1}^{5}I'_i,\ w\text{ is isolated in }G[I'-\cup_{i=1}^{5}I'_i]\},\\
I'_7&=(I'-\cup_{i=1}^{6}I'_i)\cap N(\cup_{i=1}^4 A_i),\quad I'_{8}=(I'-\cup_{i=1}^{7}I'_i)\cap N(\cup_{i=5}^{9} A_i).
\end{align*}
It follows that $I'_i$ ($i=1,2,\dots,8$) form a partition of $I'$. We further partition $I'_4$, $I'_5$, $I'_6$, $I'_7$, $I'_8$ and $I_{62}^{(3)}$ as follows.
\begin{align*}
I'_{41}&=I'_4\cap N(\cup_{i=1}^6 A_i),\quad I'_{42}=(I'_4-I'_{41})\cap N(\cup_{i=7}^{9} A_i);\\
I'_{51}&=I'_5\cap N(\cup_{i=1}^6 A_i),\quad I'_{52}=(I'_5-I'_{51})\cap N(\cup_{i=7}^{9} A_i);\\
I'_{61}&=I'_6\cap N(\cup_{i=1}^6 A_i),\quad I'_{62}=(I'_6-I'_{61})\cap N(\cup_{i=7}^{9} A_i);\\
I'_{71}&=I'_7\cap N(I'_{8}),\quad I'_{72}=\{w\mid w\in I'_7-I'_{71}\text{ and }w\text{ is isolated in }G[I'_7-I'_{71}]\},\\
I'_{73}&=\{w\mid w\in I'_7-I'_{71}\text{ and }w\text{ is not isolated in }G[I'_7-I'_{71}]\};\\
I'_{81}&=I'_{8}\cap N(I'_{7}),\quad I'_{82}=\{w\mid w\in I'_{8}-I'_{81}\text{ and }w\text{ is isolated in }G[I'_{8}-I'_{81}]\},\\
I'_{83}&=\{w\mid w\in I'_{8}-I'_{81}\text{ and }w\text{ is not isolated in }G[I'_{8}-I'_{81}]\};\\
I_{62}^{(31)}&=I_{62}^{(3)}\cap V(\rbjt{D_4}),\quad  I_{62}^{(32)}=(I_{62}^{(3)}-I_{62}^{(31)})\cap N((\cup_{i=1}^3 I'_i)\cup I'_{41}\cup I'_{51}\cup (I-I_6)),\\
I_{62}^{(33)}&=I_{62}^{(3)}-I_{62}^{(31)}-I_{62}^{(32)}.
\end{align*}


\begin{Orientation}\label{cons12}
{\noindent\rm (i)} Let $\cup_{i=1}^6 A_i\rightarrow I'_{41}\cup I'_{51}\rightarrow I$, $\cup_{i=7}^{9} A_i\rightarrow I'_{52}\rightarrow I$, $(I-I_{62}^{(3)})\rightarrow I'_{42}\rightarrow A\cup I_{62}^{(3)}$, $\cup_{i=1}^4 A_i\rightarrow I'_{71}\rightarrow I'_{81}\rightarrow A$, $I'_{71}\rightarrow I'_{72}\rightarrow (\cup_{i=1}^4 A_i)$, $\cup_{i=5}^{9} A_i\rightarrow I'_{82}\rightarrow I'_{81}$, $(\cup_{i=1}^3 I'_i)\cup (I-I_6)\rightarrow (I_{62}^{(32)}\cup I_{62}^{(33)})\rightarrow A$, $(\cup_{i=1}^4 I'_i)\cup I'_{51}\cup I'_{61}\cup  I'_7\rightarrow I'_{52}\cup I'_{62}\cup I'_8$. 

{\noindent\rm (ii)} Let $R=\cup_{i=1}^4 A_i$, $S=I'_{73}$, or $R=\cup_{i=5}^{9} A_i$, $S=I'_{83}$. Then orient the edges of $G[S]$ and $[R,S]$ by an $R$-$S$ orientation.

{\noindent\rm (iii)} For $w\in I'_6$, if $w$ is not isolated in $G[I']$, then $N(w)\cap  (\cup_{i=1}^5 I'_i)\neq\emptyset$, and let $\cup_{i=1}^6 A_i\rightarrow w\rightarrow \cup_{i=1}^5 I'_i$ if $w\in I'_{61}$, $(\cup_{i=1}^4 I'_i)\cup I'_{51}\rightarrow w\rightarrow A$ if $w\in I'_{62}$ and $N(w)\cap I'_{52}=\emptyset$, and $\cup_{i=7}^{9} A_i\rightarrow w\rightarrow  I'_{52}$ if $w\in I'_{62}$ and $N(w)\cap I'_{52}\neq\emptyset$; if $w$ is isolated in $G[I']$, then $|[w,A]|\geq 2$. Thus orient $w'w$ as $\rbjt{w'w}$ for some $w'\in N(w)\cap (\cup_{i=1}^6 A_i)$ and the other edges of $[w,A]$ away from $w$ if $w\in I'_{61}$, and orient the edges of $[w,A]$ in two ways if $w\in I'_{62}$. 
\end{Orientation}

\begin{claim}\label{cla13_1}
For $w\in I_{62}^{(3)}$, we have $\partial(w,u)\leq \begin{cases}
g^*-2, &\text{ if } w\in I_{62}^{(31)},\\
g^*, &\text{ if } w\in I_{62}^{(32)}\cup I_{62}^{(33)}.
\end{cases}$
\end{claim}

\begin{proof}
If $w\in I_{62}^{(31)}$, then by the definitions of $I_{62}^{(2)}$ and $I_{62}^{(31)}$, there exists a directed cycle of length at most $g^*$ containing $u$ and $w$ by Orientation \ref{cons8.5}, which implies $\partial(w,u)\leq g^*-2$. If $w\in I_{62}^{(32)}\cup I_{62}^{(33)}$, then $w\rightarrow A$ and $\partial(w',u)\leq g^*-1$ for any $w'\in A$ by Orientation \ref{cons12}, Claims \ref{cla1}, \ref{cla7} and \ref{cla9.5}, which implies $\partial(w,u)\leq g^*$.			
\end{proof}

For the convenience of subsequent arguments, we summarize prior results on $A$ and $I$ based on Claims \ref{cla1}, \ref{cla7}, \ref{cla9}, \ref{cla9.5} and \ref{cla13_1}.

\begin{proposition}\label{pro3}
{\rm (i)} For $w\in \cup_{i=1}^4 A_i$, we have $\partial(w,u)=1$ and $\partial(v,w)\leq 4$. \\
{\rm (ii)} For $w\in A_5\cup A_6$, we have $\partial(w,u)=1$ and $\partial(v,w)\leq 5$.\\
{\rm (iii)} For $w\in \cup_{i=7}^{9} A_i$, we have $\partial(w,u)\leq g^*-1$ and $\partial(v,w)\leq g^*+3$.
\end{proposition}

\begin{proposition}\label{pro1.5}
{\rm (i)} For $w\in I-I_6$, $\partial(w,u)\leq 2$ and $\partial(v,w)\leq 6$.\\
{\rm (ii)} For $w\in I_{61}$, $\partial(w,u)\leq g^*$ and $\partial(v,w)\leq g^*+2$. 
\\ {\rm (iii)} For $w\in I_{62}^{(1)}\cup I_{62}^{(2)}$, $\partial(w,u)\leq 2$ and $\partial(v,w)\leq g^*+2$. \\{\rm (iv)} For $w\in I_{62}^{(3)}$, $\partial(w,u)\leq g^*$. 
\end{proposition}

For $w\in I'_1\cup I'_2$, by Claim \ref{cla9.5}, we have $\partial(w,u)\leq 2$, with $\partial(v,w)\leq 5$ if $w\in I'_1$ and $\partial(v,w)\leq 6$ if $w\in I'_2$. Next, we give some directed distances with respect to $w\in \cup_{i=3}^{8}I'_i$.

\begin{claim}\label{cla14}
Table \ref{Table7} presents upper bounds on directed distances for $\cup_{i=3}^{8} I_i'$.
\begin{table}[htbp]
\centering
\setlength{\tabcolsep}{4pt}
\renewcommand{\arraystretch}{1.1}
$
\begin{array}{|c|c|c|c|c|c|c|c|c|c|}
\hline
\text{for } w \text{ in} & I'_3 & I'_{41}\cup I'_{51}\cup I'_7& I'_{42} &I'_{52}&  I'_{61}&I'_{62}  & I'_{8} \\
\hline
\partial(w,u) \le & g^*-2 & g^*+1&g^*&g^*+1&g^*+2 & g^*+2  & g^*+1 \\
\hline
\partial(v,w) \le & 2g^*-3 & 6& g^*+3&g^*+4&6 & g^*+4 & g^*+5 \\
\hline
\end{array}
$\vspace{-5pt}
\caption{Upper bounds on directed distances for $w \in \cup_{i=3}^{8}I'_i$.}
\label{Table7}
\vspace{-10pt}
\end{table}
\end{claim}

\begin{proof}
For $w \in I'_3$, we have $w\in V(\rbjt{D_4})$, and thus there is a directed cycle of length at most $g^*$ containing $u$ and $w$ by Orientation \ref{cons8.5}, which implies $\theta(u,w) \leq g^* - 2$. Therefore, the desired bounds hold since $\partial(v,u) = g^* - 1$ in Claim \ref{cla1}.

For $w\in I'_4\cup I'_5$, by the definitions of $I'_{41},I'_{42}, I'_{51}$ and $I'_{52}$, we have $N(w)\cap (I-I_{62}^{(3)})\neq\emptyset$ if $w\in I'_4$, $N(w)\cap I_{62}^{(3)}\neq\emptyset$ if $w\in I'_5$, $N(w)\cap (\cup_{i=1}^6 A_i)\neq\emptyset$ if $w\in I'_{41}\cup I'_{51}$, and $N(w)\cap (\cup_{i=7}^{9} A_i)\neq\emptyset$ if $w\in I'_{42}\cup I'_{52}$. By Orientation \ref{cons12}, we have $ \cup_{i=1}^6 A_i\rightarrow I'_{41}\cup I'_{51}\rightarrow I$, $(I-I_{62}^{(3)})\rightarrow I'_{42}\rightarrow A$, and $\cup_{i=7}^{9} A_i\rightarrow I'_{52}\rightarrow I$, which implies the desired result holds by Propositions \ref{pro3} and \ref{pro1.5}.

For $w\in  I'_6$, by the definition of $I'_{61}$ and $I'_{62}$, we have $N(w)\cap (\cup_{i=1}^6 A_i)\neq\emptyset$ if $w\in I'_{61}$, and $N(w)\cap (\cup_{i=7}^{9}A_i)\neq\emptyset$ if $w\in I'_{62}$. By Orientation \ref{cons12}, if $w$ is not isolated in $G[I']$, then $N(w)\cap  (\cup_{i=1}^5 I'_i)\neq\emptyset$,  $\cup_{i=1}^6 A_i\rightarrow w\rightarrow \cup_{i=1}^5 I'_i$ if $w\in I'_{61}$, and $(\cup_{i=1}^4 I'_i)\cup I'_{51}\rightarrow w\rightarrow A$ or $\cup_{i=7}^{9} A_i\rightarrow w\rightarrow  I'_{52}$ if $w\in I'_{62}$, and thus the desired bounds hold by Claim \ref{cla9.5} and Proposition \ref{pro3}; if $w$ is isolated in $G[I']$, then $|[w,A]|\geq 2$, and thus $w'\rightarrow w\rightarrow w''$ for some $w'\in N(w)\cap (\cup_{i=1}^6 A_i)$ and $w''\in N(w)\cap A$ if $w\in I'_{61}$, and for some $w',w''\in N(w)\cap A$ if $w\in I'_{62}$, which implies the desired bounds hold by Proposition \ref{pro3}. 

For $w\in I'_7\cup I'_8$, we have $N(w)\cap I'_{71}\neq\emptyset$ for $w\in I'_{81}$, $N(w)\cap I'_{81}\neq\emptyset$ for $w\in I'_{71}$, $N(w)\cap (\cup_{i=1}^4 A_i)\neq\emptyset$ for $w\in I'_7$, and $N(w)\cap (\cup_{i=5}^{9} A_i)\neq\emptyset$ for $w\in I'_{8}$. By the definitions of $I'_6,I'_{71}$ and $I'_{81}$, we have $N(w)\cap (I'_7\cup I'_{8})\neq\emptyset$ if $w\in I'_{72}\cup I'_{82}$, which implies $N(w)\cap I'_{71}\neq\emptyset$ if $w\in I'_{72}$, and $N(w)\cap I'_{81}\neq\emptyset$ if $w\in I'_{82}$. By Orientation \ref{cons12}, we have $\cup_{i=1}^4 A_i\rightarrow I'_{71}\rightarrow I'_{81}\rightarrow A$, $I'_{71}\rightarrow I'_{72}\rightarrow \cup_{i=1}^4 A_i$ and $\cup_{i=5}^{9} A_i\rightarrow I'_{82}\rightarrow I'_{81}$, which implies $\partial(w,u)\leq g^*+1,\partial(v,w)\leq 6$ if $w\in I'_{71}\cup I'_{72}$, and $\partial(w,u)\leq g^*+1,\partial(v,w)\leq g^*+4 $ if $w\in I'_{81}\cup I'_{82}$ by Proposition \ref{pro3}. By Orientation \ref{cons12}, Proposition \ref{pro3} and Lemma \ref{l1}, we obtain the result of $w\in I'_{73}\cup I'_{83}$.
\end{proof}

\begin{claim}\label{cla15}
For $w\in I_{62}^{(31)}\cup I_{62}^{(32)} $, we have $\partial(v,w)\leq \begin{cases}
2g^*-3, &\text{ if } w\in I_{62}^{(31)},\\
g^*+3, &\text{ if } w\in I_{62}^{(32)}.
\end{cases}$
\end{claim}

\begin{proof}
If $w\in I_{62}^{(31)}$, then $w\in V(\rbjt{D_4})$, and thus there exists a directed cycle of length at most $g^*$ containing $u,w$ by Orientation \ref{cons8.5}. Hence we have $\partial(u,w)\leq g^*-2$, which implies $\partial(v,w)\leq 2g^*-3$ by $\partial(v,u)=g^*-1$ in Claim \ref{cla1}.

If $w\in I_{62}^{(32)}$, then by the definition of $I_{62}^{(32)}$ and Orientation \ref{cons12}, $N(w)\cap((\cup_{i=1}^3 I'_i)\cup I'_{41}\cup I'_{51}\cup (I-I_6))\neq\emptyset$ and $(\cup_{i=1}^3 I'_i)\cup I'_{41}\cup I'_{51}\cup (I-I_6)\rightarrow w$, which implies $\partial(v,w)\leq g^*+3$ by Proposition \ref{pro1.5}, Claims \ref{cla9.5} and \ref{cla14}.
\end{proof}

To orient the edges of $[K',I\cup K]$, we partition $K'$ as follows.
\begin{align*}
K'_1&=K'\cap V(\rbjt{D_3}),\quad K'_2=\{w\mid w\in K'-K'_1  ,\   d(w,B)=3  \},\\
K'_3&=\{w\mid w\in K'-K'_1-K'_2,\ N(w)\cap (K'_1\cup K'_2)\neq\emptyset\},\\  K'_4&=(K'-\cup_{i=1}^3 K'_i)\cap N(K), \quad K'_5=(K'-\cup_{i=1}^4 K'_i) \cap N(I-I_{62}^{(3)})\cap N(I_{62}^{(33)}),\\ K'_6&=(K'-\cup_{i=1}^5 K'_i) \cap N(I_{62}^{(31)}\cup I_{62}^{(32)})\cap N(I_{62}^{(33)}),\\
K'_7&=\{w\mid w\in K'-\cup_{i=1}^6 K'_i ,\  w \text{ is isolated in }G[K'-\cup_{i=1}^6 K'_i]\},\\
K'_8&=(K'-\cup_{i=1}^7 K'_i)\cap N(I-I_{62}^{(3)}),\ K'_9=(K'-\cup_{i=1}^8K'_i)\cap N(I_{62}^{(31)}\cup I_{62}^{(32)}),\\
K'_{10}&=(K'-\cup_{i=1}^9K'_i)\cap N(I_{62}^{(33)}).
\end{align*}

We further partition $K'_2$, $K'_8$ and $K'_9$ as follows.
\begin{align*}
K'_{21}&=K'_2\cap N(K'_3),\quad K'_{22}=\{w\mid w\in K'_2-K'_{21},\ w\text{ is isolated in }G[K'_2-K'_{21}]\},\\
K'_{23}&=\{w\mid w\in K'_2-K'_{21},\ w\text{ is not isolated in }G[K'_2-K'_{21}]\};\\
K'_{81}&=K'_8\cap N(K'_9),\quad K'_{82}=\{w\mid w\in K'_8-K'_{81},\ w\text{ is isolated in }G[K'_8-K'_{81}]\},\\
K'_{83}&=\{w\mid w\in K'_8-K'_{81},\ w\text{ is not isolated in }G[K'_8-K'_{81}]\};\\
K'_{91}&=K'_9\cap N(K'_8),\quad K'_{92}=\{w\mid w\in K'_9-K'_{91},\ w\text{ is isolated in }G[K'_9-K'_{91}]\},\\
K'_{93}&=\{w\mid w\in K'_9-K'_{91},\ w\text{ is not isolated in }G[K'_9-K'_{91}]\};
\end{align*}

\begin{proposition}\label{pro2}
{\rm (i)} For $w\in K'$, if $N(w)\cap I\subseteq I_{62}^{(33)}$, then $w\in \cup_{i=1}^4 K'_i$. Thus $K'_{10}=\emptyset$. \\		
{\rm (ii)} For $w\in K'$, if $N(w)\cap I\subseteq I_{62}^{(3)}$ and $g^*=5$, then $w\in \cup_{i=1}^4 K'_i$. Moreover, $K'_6\cup K'_{9}\cup K'_{10}=\emptyset$ if $g^*=5$.
\end{proposition}
\begin{proof}
(i) Let $w\in K'$. By the definitions of $I_{62}^{(1)}$, $I_{62}^{(32)}$ and $I_{61}$, we have $N(I_{62}^{(33)})\cap ((\cup_{i=1}^3 I'_i)\cup I'_{41}\cup I'_{51}\cup (I-I_6)\cup I_{61}\cup (\cup_{i=1}^3 A_i))=\emptyset$. Thus there is no $(w,B)$-path $wx_1x_2x_3x_4$ such that
$x_1\in I_{62}^{(33)}$, $x_2\in I'\cup I\cup A$, $x_3\in A$ and $x_4\in B$.
Indeed, otherwise we would have $x_3\in A_1$ and $x_2\in (\cup_{i=1}^3 I'_i)\cup I'_{41}\cup I'_{51}\cup (I-I_6)\cup I_{61}\cup (\cup_{i=1}^3 A_i)$, a contradiction with $x_2\in N(x_1)$. Clearly, we have $N(I_{62}^{(33)})\cap (I_1\cup S_{2,2}\cup  S_{3,3})=\emptyset$. Thus there is no $(w,B)$-path $wy_1y_2y_3y_4$ satisfying $y_1\in I_{62}^{(33)}$, $y_2\in I\cup S_{3,3}$, $y_3\in S_{2,2}\cup J$ and $y_4\in B$, since necessarily $y_2\in I_1\cup S_{3,3}$. Since $d(G)=4$ and $N(w)\cap I\subseteq I_{62}^{(33)}$, by checking every shortest $(w,B)$-path, we conclude that $N(w)\cap (K'_1\cup K'_2\cup K)\neq\emptyset$, which yields $w\in \cup_{i=1}^4 K'_i$. For any $w\in K'_{10}$, we have $N(w)\cap I\subseteq I_{62}^{(33)}$ by the definitions of $K'_8$ and $K'_{9}$. Consequently, $w\in \cup_{i=1}^4 K'_i$, which gives a contradiction. Therefore, (i) holds.

{\noindent (ii)} Since $g^*=5$, we have $A_1=B_1=\emptyset$. Clearly, we have $N(I_{62}^{(3)})\cap (A_2\cup I_1\cup S_{2,2}\cup S_{3,3})=\emptyset$. Since $d(G)=4$ and $N(w)\cap I\subseteq I_{62}^{(3)}$, by checking every shortest $(w,B)$-path and the definition of $I_6$, we conclude that $N(w)\cap (K'_1\cup K'_2\cup K)\neq\emptyset$, which yields $w\in \cup_{i=1}^4 K'_i$. For any $w\in K'_6\cup K'_9\cup K'_{10}$, we have $N(w)\cap I\subseteq I_{62}^{(3)}$. Consequently, if $g^*=5$, then $w\in \cup_{i=1}^4 K'_i$, which gives a contradiction. Hence (ii) holds. 	  
\end{proof}

\vspace{10pt}
To control the directed distance of the vertices in $K'_2$, we define a notation:\vspace{-10pt} $$I^*=\{w\mid w\in I, \exists\  w'\in K'_2 \text{ such that there exists a shortest }(w',B)\text{-path}\\ \text{ containing }w \}.\vspace{-10pt}$$ It follows that $I^*\subseteq I-I_{62}$ and $N(w)\cap I^*\neq\emptyset$ for $w\in K'_2$. 

\begin{Orientation}\label{cons11}
{\noindent\rm (i)}	Let $K'\rightarrow I_{62}^{(33)}$, $I^*\rightarrow K'_{21}$, $K'_1\cup K'_{21}\rightarrow K'_3\rightarrow I$, $K\rightarrow K'_2\cup K'_3$, $K\rightarrow K'_4\rightarrow I$, $(I-I_{62}^{(3)})\rightarrow K'_5$, $(I_{62}^{(31)}\cup I_{62}^{(32)})\rightarrow K'_6$, $(I-I_{62}^{(3)})\rightarrow K'_{81}\rightarrow K'_{91}\rightarrow (I_{62}^{(31)}\cup I_{62}^{(32)})$, $ K'_{81}\rightarrow K'_{82}\rightarrow (I-I_{62}^{(3)})$, $ (I_{62}^{(31)}\cup I_{62}^{(32)})\rightarrow K'_{92}\rightarrow K'_{91}$.  

{\noindent\rm (ii)} For $w\in K'_{22}$, if $w$ is not isolated in $G[K']$, then by the definitions of $K'_{21}$ and $K'_3$, $N(w)\cap (K'_1\cup K'_{21})\neq\emptyset$. Let $I^*\rightarrow w\rightarrow K'_1$ if $N(w)\cap K'_1\neq \emptyset$, and  $K'_{21}\rightarrow w\rightarrow I^*$ otherwise. Now we assume $w$ is isolated in $G[K']$. If $N(w)\cap K\neq\emptyset$, then let 
$K\rightarrow w\rightarrow I^*$.
If $N(w)\cap K=\emptyset$, then $|[w,I]|\geq 2$, let $w'\rightarrow w$ for some $w'\in I^*$, and the other edges of $[w,I]$ away from $w$.

{\noindent\rm (iii)} For $w\in K'_7$, $N(w)\cap (I-I_{62}^{(33)})\neq\emptyset$ by (i) of Proposition \ref{pro2}. If $w$ is not isolated in $G[K']$, then $N(w)\cap \left(\cup_{i=3}^6 K'_i\right)\neq\emptyset$ by the definition of $K'_3$, and thus let $(I-I_{62}^{(33)})\rightarrow w\rightarrow \cup_{i=3}^6 K'_i$. Now we assume $w$ is isolated in $G[K']$. Then by the definition of $K'_4$, we have $N(w)\cap K=\emptyset$, and thus $|[w,I]|\geq 2$. If $N(w)\cap (I-I_{62}^{(3)})\neq\emptyset$, then orient $w'w$ as $w'\rightarrow w$ for some $w'\in N(w)\cap (I-I_{62}^{(3)})$ and the other edges of $[w,I]$ away from $w$. Otherwise, we orient the edges of $[w,I]$ in two ways.

{\noindent\rm (iv)} Let $R=I^*,S=K'_{23}$, $R=I-I_{62}^{(3)},S=K'_{83}$, or $R=I_{62}^{(31)}\cup I_{62}^{(32)},S=K'_{93}$. Then we give an $R$-$S$ orientation for the edges of $G[S]$ and $[R,S]$. 


{\noindent\rm (v)} For all edges not oriented in $G$, we orient these edges arbitrarily. 
\end{Orientation}

At this point, we have constructed an orientation $\rbjt{G}$ of $G$. 

\begin{proposition}\label{pro5}
For any $w\in I^*$, we have $\partial(w,u)\leq 4$ and $\partial(v,w)\leq 6$.
\end{proposition}

\begin{proof}
Let $w\in I^*$. Then $I^*\subseteq I-I_{62}$ and there exists $w'\in K'_2$ and a shortest $(w',B)$-path $P$ containing $w$. Clearly, $I-I_{62}=(I-I_6)\cup I_{61}$. If $w\in I_{61}$, then $V(P)\cap A_1\neq\emptyset$, and thus $A_1\neq\emptyset$, which implies $g^*=4$, $\partial(w,u)\leq g^*=4$ and $\partial(v,w)\leq g^*+2=6$ by Proposition \ref{pro1.5}. If $w\in I-I_6$, then $\partial(w,u)\leq 2$ and $\partial(v,w)\leq 6$ by Proposition \ref{pro1.5}.
\end{proof}

\begin{claim}\label{cla12}
Table \ref{Table5} presents upper bounds on directed distances for $K_i'$.

\begin{table}[H]
\centering
\renewcommand{\arraystretch}{1.1}
\setlength{\tabcolsep}{4pt}
\begin{tabular}{|l|*{9}{>{\centering\arraybackslash}p{1.35cm}|}}
\hline
\textbf{for \(w\) in} &$K'_1$&$K'_2\cup K'_7$&$K'_3$&$K'_4$&$K'_5$& $K'_6$&$K'_8$&$K'_9$\\
\hline
$\partial(w,u)\le$ &$g^*-1$&$g^*+2$&$g^*+1$&$g^*+1$&$g^*+1$&$5$&$g^*+2$&$g^*+2$\\
\hline
$\partial(v,w)\le$ &$g^*+1$&$8$&$8$&$6$&$g^*+3$&$8$&$g^*+4$&$9$\\
\hline
\end{tabular}
\vspace{-5pt}
\caption{Upper bounds on distances for $w\in K'_{i}$.}
\label{Table5}
\vspace{-15pt}
\end{table}
\end{claim}

\begin{proof}
If $w\in K'_1$, then by Orientation \ref{cons8}, there exists a directed path (or cycle) $x_1x_2wx_3x_4x_5$, $x_1x_2x_3wx_4x_5$, or $y_1y_2wy_3y_4$ such that $x_1,y_1\in \cup_{i=1}^3 A_i$ and $x_5,y_4\in A$. Moreover, if \( g^* = 4 \), there exists only \( y_1 y_2 w y_3 y_4 \). By the definition of $I_{62}^{(2)}$, we have $x_4,y_3\notin I_{62}^{(3)}$, which implies $x_5,y_4\in A-A_{94}-A_{95}$. Hence the conclusion follows from Claim \ref{cla1} and the fact that $(A-A_{94}-A_{95})\rightarrow u$ given in Orientation \ref{cons8.5}.

If $w\in K_2'$, we have $N(w)\cap I^*\neq\emptyset$.
For $w\in K'_{21}$, we have $N(w)\cap K'_3\neq\emptyset$. By Orientation \ref{cons11}, $I^*\rightarrow w\rightarrow K'_3\rightarrow I$. Then $\partial(w,u)\leq \partial(w,I)+\partial(I,u)\leq 2+g^*$ by Proposition \ref{pro1.5}, and $\partial(v,w)\leq \partial(v,I^*)+\partial(I^*,w)\leq 6+1= 7$ by Proposition \ref{pro5}. For $w\in K'_{22}$, if $w$ is not isolated in $G[K']$, then $\partial(w,u)\leq g^*$ and $\partial(v,w)\leq 7$ if $N(w)\cap K'_1\neq\emptyset$, and $\partial(w,u)\leq 5$ and $\partial(v,w)\leq 8$ otherwise by Orientation \ref{cons11} and Proposition \ref{pro5}; if $w$ is isolated in $G[K']$, then the expected result also holds by Claim \ref{cla7}. By Lemma \ref{l1}, we have $\partial(w,u)\leq 6$ and $\partial(v,w)\leq 8$ for $w\in K'_{23}$.

If $w\in K'_3$, then $\partial(w,u) \leq\partial(w,I)+\partial(I,u)\leq  1+ g^*$ by Orientation \ref{cons11} and Proposition \ref{pro1.5}.
By the definition of $K'_3$, $N(w)\cap (K'_1\cup K'_{21})\neq\emptyset$, and thus $\partial(v,w)\leq \partial(v,w_1)+\partial(w_1,w)\leq 7+1=8$ for $w_1\in K'_1\cup K'_{21}$ by the result of $K'_1\cup K'_{21}$. 

If $w\in K'_4\cup K'_5\cup K'_6$,	by Orientation \ref{cons11}, Claim \ref{cla7} and Proposition \ref{pro1.5}, we have $\partial(w,u)\leq g^*+1$ and $\partial(v,w)\leq 6$ if $w\in K'_4$, and $\partial(w,u)\leq g^*+1$ and $\partial(v,w)\leq g^*+3$ if $w\in K'_5$. For $w\in K'_6$, we have $g^*=4$ by (ii) of Proposition \ref{pro2}, which implies $\partial(w,u)\leq g^*+1=5$ and $\partial(v,w)\leq g^*+4=8$ by Proposition \ref{pro1.5} and Claim \ref{cla15}. 

If $w\in K'_7$, we have $N(w)\cap (I-I_{62}^{(33)})\neq\emptyset$ for $g^*=4$, and $N(w)\cap (I-I_{62}^{(3)})\neq\emptyset$ for $g^*=5$ by Proposition \ref{pro2}. By the definitions of $K'_5$ and $K'_6$, we have $N(w)\cap I_{62}^{(33)}=\emptyset$. Whether $w$ is an isolated vertex in $G[K']$ or not, by Orientation \ref{cons11} and Proposition \ref{pro1.5}, if $g^*=5$, then $(I-I_{62}^{(3)})\rightarrow w\rightarrow \cup_{i=3}^6 K'_i$, or $w'\rightarrow w\rightarrow w''$ for some $w'\in I-I_{62}^{(3)},w''\in I-I_{62}^{(33)}$ since $N(w)\cap I_{62}^{(33)}=\emptyset$, which implies $\partial(w,u)\leq g^*+2=7$ and $\partial(v,w)\leq g^*+3=8$. Similarly, if $g^*=4$, then $\partial(w,u)\leq 6$ and $\partial(v,w)\leq 8$ by Claim \ref{cla15}.

If $w\in K'_{81}\cup K'_{82}\cup K'_{91}\cup K'_{92}$, we have $N(w)\cap K'_{81}\neq \emptyset$ for $w\in K'_{82}\cup K'_{91}$, and $N(w)\cap K'_{91}\neq\emptyset$ for $w\in K'_{81}\cup K'_{92}$ by the definitions of $ K'_7$. By Orientation \ref{cons11}, we have $(I-I_{62}^{(3)})\rightarrow K'_{81}\rightarrow K'_{91}\rightarrow (I_{62}^{(31)}\cup I_{62}^{(32)})$, $ K'_{81}\rightarrow K'_{82}\rightarrow (I-I_{62}^{(3)})$, $ (I_{62}^{(31)}\cup I_{62}^{(32)})\rightarrow K'_{92}\rightarrow K'_{91}$, which implies that $\partial(w,u)\leq g^*+2$ and $\partial(v,w)\leq g^*+4$ by Proposition \ref{pro1.5} and Claim \ref{cla15}.

If $w\in K'_{83}\cup K'_{93}$, we consider the following two cases. For $w\in K'_{83}$, by Lemma \ref{l1}, Proposition \ref{pro1.5}, Claim \ref{cla15} and Orientation \ref{cons11}, we have $\partial(w,u)\leq g^*+2$ and $\partial(v,w)\leq g^*+4$. For $w\in K'_{93}$, by (ii) of Proposition \ref{pro2}, we have $g^*=4$. Similarly, $\partial(w,u)\leq 6$ and $\partial(v,w)\leq 9$.
\end{proof}
\begin{claim}\label{cla13}
For $w\in I_{62}^{(33)}$, we have $\partial(v,w)\leq 9$. 
\end{claim}

\begin{proof}
By the definition of $ I_{62}^{(33)}$, we have $N(w)\cap K'\neq\emptyset$. It is easy to check that $N(w)\cap (\cup_{i=8}^{10} K'_i)=\emptyset$ by (i) of Proposition \ref{pro2}, the definitions of $K'_5$ and $K'_6$. By Orientation \ref{cons11}, we have $\cup_{i=1}^{7} K'_i \rightarrow w$, which implies $\partial(v,w)\leq 9$ by Claim \ref{cla12}.
\end{proof}

\vspace{5pt}

We are now prepared to prove Claim \ref{cla19}.

{\noindent\textbf{\bf {Proof of (i) of Claim \ref{cla19}.}}} By Tables \ref{Table8-1} and \ref{Table8-2}, we have $\partial(x,y)\leq \partial(x,u)+\partial(u,v)+\partial(v,y)\leq g^*+13$ if $y\in B\cup J$, or $x\in B\cup J$ and $y\in A\cup A'\cup B'\cup J'\cup K\cup L\cup L'\cup S_{2,2}\cup S_{3,3}\cup M$, or $x\in J$ and $y\in M'$. Next we consider $x\in B\cup J, y\in I\cup I'\cup K'$, and $x\in B,y\in M'$.

\vspace{0.1\baselineskip}
{\noindent\bf Case 1.} $x\in B$ and $y\in I$.
\vspace{0.1\baselineskip}		

If $x\notin  B_8\cup B_9$, then by Claims \ref{cla1}, \ref{cla7}, \ref{cla7_1} and \ref{cla10}, $\partial(x,u)\leq 2g^*-2$, and thus $\partial(x,y)\leq 2g^*+8\leq g^*+13$ by Table \ref{Table8-1}. If $y\notin I_{62}^{(33)}$, then by Proposition \ref{pro1.5} and Claim \ref{cla15}, we have $\partial(v,y)\leq g^*+3$, and thus $\partial(x,y)\leq g^*+13$ by Table \ref{Table8-1}.

Now we assume $x\in B_8\cup B_9$ and $y\in I_{62}^{(33)}$. Then $N(y)\cap ((\cup_{i=1}^3 A_i)\cup (\cup_{i=1}^3 I'_i)\cup I'_{41}\cup I'_{51}\cup (I-I_6)\cup I_{61}\cup K)=\emptyset$ by the definitions of $I_{61}$, $I_{62}^{(32)}$ and $I_{62}^{(1)}$. Clearly, $N(I'-(\cup_{i=1}^3 I'_i)- I'_{41}- I'_{51})\cap I\cap (\cup_{i=1}^6 A_i)=\emptyset$. Let $x'\in N(x)\cap (J_5\cup J_6)$. Then $N(x')\cap (J_1\cup J_2\cup L_1)=\emptyset$ by the definitions of $J_3$ and $J_4$. Checking all $(x',y)$-paths, $d(x',y)\geq 5$, a contradiction with $d(G)=4$.

\vspace{0.1\baselineskip}
{\noindent \rm \bf Case 2.} $x\in B$ and $y\in I'$.
\vspace{0.1\baselineskip}	

If $x\notin  B_8\cup B_9\cup B_{10}^{(1)}$, then by Claims \ref{cla1}, \ref{cla7}, \ref{cla7_1} and \ref{cla10}, $\partial(x,u)\leq 6$, and thus $\partial(x,y)\leq g^*+12$ by Table \ref{Table8-1}. If $y\notin  I'_{52}\cup I'_{62}\cup I'_{8}$, then by Claims \ref{cla9.5} and \ref{cla14}, $\partial(v,y)\leq g^*+3$, and thus $\partial(x,y)\leq g^*+13$.

Now we assume $x\in  B_8\cup B_9\cup B_{10}^{(1)}$ and $y\in I'_{52}\cup I'_{62}\cup I'_{8}$. Then $N(y)\cap ((\cup_{i=1}^4 A_i)\cup (I-I_{62}^{(3)}))=\emptyset$. Let $x'\in N(x)\cap (J_5\cup J_6\cup J')$. Then $d(x',y)=4$, and thus we take a shortest $(x',y)$-path $P=x'x_1x_2x_3y$. Then $x_1\in B_1\cup J_1$, $x_2\in A_1\cup I_1$ and $x_3\in I'$, and thus $x_3\in (\cup_{i=1}^4 I'_i)\cup I'_{51}\cup I'_{61}\cup I'_{7}$ since $N(x_3)\cap (A_1\cup I_1)\neq\emptyset$, and $x_3\rightarrow y$ by Orientation \ref{cons12}. If $x_3\notin I'_{42}$, then by Claims \ref{cla9.5} and \ref{cla14}, we have $\partial(v,x_3)\leq g^*+2$. If $x_3\in I'_{42}$, then $N(x_3)\cap A_1=\emptyset$, and thus $N(x_3)\cap I_1\neq\emptyset$, which implies $\partial(v,x_3)\leq\partial(v,I_1)+\partial(I_1,x_3)\leq 4$ by Claim \ref{cla1} and Orientation \ref{cons12}. Then $\partial(v,y)\leq \partial(v,x_3)+\partial(x_3,y)\leq g^*+3$ by Orientation \ref{cons12}, which implies $\partial(x,y)\leq g^*+13$.

\vspace{0.1\baselineskip}
{\noindent \rm \bf Case 3.} $x\in B$ and $y\in K'$.
\vspace{0.1\baselineskip}

If $x\notin  B_8\cup B_9$, then by Claims \ref{cla1}, \ref{cla7}, \ref{cla7_1} and \ref{cla10}, $\partial(x,u)\leq 2g^*-2$, and thus $\partial(x,y)\leq 2g^*+8$ by Table \ref{Table8-2}. 

Now we assume $x\in B_8\cup  B_9$. Then $N(x)\cap J\subseteq J_5\cup J_6$, $N(x)\cap B_1=\emptyset$ and $d(x,y)=4$. Let $P=xx_1x_2x_3y$ be a shortest $(x,y)$-path. Then $x_1\in B_2\cup B_3\cup B_4\cup J'$, $x_2\in S_{2,2}\cup J_1$ and $x_3\in I_1$. 

If $x_1\in B_2\cup B_3\cup B_4$, then by Orientation \ref{cons1} and Claim \ref{cla1}, we have $x\rightarrow x_1$ and $\partial(x_1,u)\leq 4$, and thus $\partial(x,u)\leq 5$. Hence $\partial(x,y)\leq 15$ by Table \ref{Table8-2}. 


If $x_1\in J'$, then $x_2\in J_1$ and $x_3\in I_1$, so that $x_1\in J_1'\cup J_2'$. By Orientation \ref{cons81}, we have $x\to x_1$. Applying Claim \ref{cla9_1}, we obtain $\partial(x,u)\leq \partial(x,x_1)+\partial(x_1,u)\leq 1+2g^*-3 = 2g^*-2$, which yields $\partial(x,y)\leq 2g^*+8\leq g^*+13$.

\vspace{0.1\baselineskip}
{\noindent \rm \bf Case 4.} $x\in J$ and $y\in I$.
\vspace{0.1\baselineskip}

If $y\notin I_{62}^{(33)}$, then by Proposition \ref{pro1.5} and Claim \ref{cla15}, we have $\partial(v,y)\leq g^*+3$, and thus $\partial(x,y)\leq g^*+12$ by Table \ref{Table8-2}. If $x\notin J_5$, then by Claims \ref{cla1}, \ref{cla7} and \ref{cla10}, we have $\partial(x,u)\leq 7$, and thus $\partial(x,y)\leq 17$.

Now we assume $y\in I_{62}^{(33)}$ and $x\in  J_5$. Then $N(x)\cap (J_1\cup J_2\cup L_1)=\emptyset$ and $N(y)\cap ((\cup_{i=1}^3 A_i)\cup  I_{61}\cup (\cup_{i=1}^3 I'_i)\cup I'_{41}\cup I'_{51}\cup (I-I_6)\cup K)=\emptyset$ by the definition of $J_5,I_{61},I_{62}^{(1)},I_{62}^{(32)}$, and thus $d(x,y)> 4$, a contradiction.

\vspace{0.1\baselineskip}
{\noindent \rm \bf Case 5.} $x\in J$ and $y\in I'$.
\vspace{0.1\baselineskip}

If $x\notin J_5$, then $\partial(x,u)\leq 7$ by Claims \ref{cla1}, \ref{cla7} and \ref{cla10}, and thus $\partial(x,y)\leq g^*+13$. If $y\notin I'_{8}$, then by Claims \ref{cla9.5} and \ref{cla14}, $\partial(v,y)\leq g^*+4$, and thus $\partial(x,y)\leq g^*+13$.

Now we assume $x\in J_5$ and $y\in I'_{8}$. Then $N(x)\cap (J_1\cup J_2\cup L_1)=\emptyset$, and $N(y)\cap ((\cup_{i=1}^4A_i)\cup I)=\emptyset$ by the definition of $I'_{8}$, and thus $d(x,y)=4$. Let $P=xx_1x_2x_3y$ be a shortest $(x,y)$-path. Then $x_1\in B_1$, $x_2\in A_1$ and $x_3\in (\cup_{i=1}^3 I'_i)\cup I'_{41}\cup I'_{51}\cup I'_{61}\cup I'_{7}$. By Orientation \ref{cons12} and Claim \ref{cla14}, we have $x_3\rightarrow y$ and $\partial(v,x_3)\leq g^*+2$, which implies $\partial(v,y)\leq g^*+3$. Hence $\partial(x,y)\leq g^*+12$ by Table \ref{Table8-1}.

\vspace{0.1\baselineskip}
{\noindent \rm \bf Case 6.} $x\in J$ and $y\in K'$.
\vspace{0.1\baselineskip}

If $x\notin J_5$, then $\partial(x,u)\leq 7$ by Claims \ref{cla1}, \ref{cla7} and \ref{cla10}, which implies $\partial(x,y)\leq 17$ by Table \ref{Table8-2}. If $y\notin K'_9$, then $\partial(v,y)\leq g^*+4$ by Claim \ref{cla12}, which implies $\partial(x,y)\leq g^*+13$.

Now we assume $x\in J_5$ and $y\in K'_9$. Then $N(x)\cap (J_1\cup J_2\cup L_1)=\emptyset$ and $N(y)\cap (K\cup (I-I_{62}^{(3)}))=\emptyset$, and thus $d(x,y)>4$, a contradiction with $d(G)=4$.

\vspace{0.1\baselineskip}
{\noindent \rm \bf Case 7.} $x\in B$ and $y\in M'$.
\vspace{0.1\baselineskip}

If $x\notin  B_8$, then by Claims \ref{cla1}, \ref{cla7}, \ref{cla7_1} and \ref{cla10}, $\partial(x,u)\leq 8$, which implies $\partial(x,y)\leq 17$ by Table \ref{Table8-2}. If $y\notin M'_2$, then by Claim \ref{cla6}, we have $\partial(v,y)\leq g^*$, which implies $\partial(x,y)\leq g^*+10$. 

Now we assume $x\in B_8$ and $y\in M'_2$. Then $N(x)\cap (A\cup S_{2,2}\cup B_1\cup (\cup_{i=1}^4 J_i))=\emptyset$, and $N(y)\cap (K\cup L)=\emptyset$ by Proposition \ref{pro1}. Then $d(x,y)=4$. Let $P=xx_1x_2x_3y$ be a shortest $(x,y)$-path. Then $x_1\in  J_5\cup J'\cup (\cup_{i=2}^5 B_i)$, $x_2\in (L-L_1)\cup S_{2,2}\cup J_1\cup J_2$ and $x_3\in M\cup S_{3,3}$.

If $x_1\in J_5$, then $x_2\in L-L_1$. By the definition of $X'$ and Proposition \ref{pro1}, we have $x_3\in X\cup M$. By Orientation \ref{cons1}, $x\rightarrow x_1\rightarrow x_2\rightarrow x_3$, and thus $\partial(x,x_3)\leq 3$. For $x_3\in X$, we have $\partial(x,u)\leq \partial(x,x_3)+\partial(x_3,u)\leq 3+5=8$ by Claim \ref{cla4}, which implies $\partial(x,y)\leq 17$. For $x_3\in M$, we have $y\in M'_{21}$, and $x_3\rightarrow y$ by Orientation \ref{cons6}, and thus $\partial(x,y)\leq \partial(x,x_3)+\partial(x_3,y)\leq 4$.

If $x_1\in J'$, then $x_2\in J_1\cup J_2$, and thus $x_1\in J'_1\cup J'_2$. By Orientation \ref{cons81}, we have $x\to x_1$. By Claim \ref{cla9_1}, we obtain $\partial(x,u)\leq \partial(x,x_1)+\partial(x_1,u)\leq  2g^*-2$, which yields $\partial(x,y)\leq 2g^*+7\leq g^*+12$.

If $x_1\in \cup_{i=2}^{5}B_i$, then $x\rightarrow x_1$ and $\partial(x_1,u)\leq 5$ by Orientation \ref{cons1}, Claims \ref{cla1} and \ref{cla7}. Thus $\partial(x,u)\leq 6$, which implies $\partial(x,y)\leq 15$.	{\hfill $\blacksquare$\par}

\vspace{0.3\baselineskip}

{\noindent\textbf{\bf {Proof of (ii) of Claim \ref{cla19}.}}} By (i) of Claim \ref{cla19}, we may assume $\{x,y\}\cap (B\cup J)=\emptyset$. By Tables \ref{Table8-1} and \ref{Table8-2}, we have $\partial(x,y)\leq \partial(x,u)+\partial(u,v)+\partial(v,y)\leq g^*+13$ if $x\in I'$, or $y\in J'$, or $y\in I'$ and $x\in A\cup A'\cup B'\cup I\cup K\cup L\cup  K'\cup L'\cup S_{2,2}\cup S_{3,3}\cup M$, or $x\in J'$ and $y\in A'\cup B'\cup K\cup L\cup L'\cup S_{2,2}$. Next we consider $y\in I'$ and $x\in  J'\cup M'$, and $x\in J'$ and $y\in A\cup I\cup K'\cup S_{3,3}\cup M\cup  M'$.

\vspace{0.1\baselineskip}
{\noindent \rm \bf Case 1.} $y\in I'$ and $x\in J'$.
\vspace{0.1\baselineskip}

If $x\notin J'_{32}\cup J'_{42}\cup J'_6$, then by Claim \ref{cla9_1}, we have $\partial(x,u)\leq 7$, and thus $\partial(x,y)\leq g^*+13$. 

Now we assume $x\in J'_{32}\cup J'_{42}\cup J'_6$. Then $N(x)\cap ((\cup_{i=1}^6 B_i)\cup (\cup_{i=1}^3 J_i))=\emptyset$, and thus $d(x,y)=4$. Let $P=xx_1x_2x_3y$ be a shortest $(x,y)$-path. Then $x_1\in J'\cup (\cup_{i=4}^6 J_i)$, $x_2\in B_1\cup  J_1$ and $x_3\in A_1\cup I_1$. If $x_1\in \cup_{i=4}^6 J_i$, then $x_2\in B_1$ and $x_3\in A_1$, and thus $x\in J'_{32}$ and $y\notin  I'_{42}\cup I'_{52}\cup I'_{62}\cup I'_8$ by $N(x)\cap (\cup_{i=4}^6 J_i)\neq\emptyset$ and $N(y)\cap A_1\neq\emptyset$, which implies $\partial(x,u)\leq 9$ and $\partial(v,y)\leq g^*+2$ by Claims \ref{cla9_1}, \ref{cla9.5} and \ref{cla14}. Hence $\partial(x,y)\leq g^*+12$. If $x_1\in J'$, then $x_2\in B_1\cup J_1$ and $x_3\in A_1\cup I_1$.

If $x_3\in I_1$, then $x_2\in J_1$, and $y\in \cup_{i=1}^4 I'_i$ since $N(y)\cap I_1\neq\emptyset$, and thus $\partial(v,y)\leq g^*+3$ by Claims \ref{cla9.5} and \ref{cla14}. Since $x_1\in J'$ and $N(x_1)\cap J_1\neq\emptyset$, it follows that $x_1\in J_1'\cup J_2'$. By Orientation \ref{cons81}, we have $x\to x_1$, and by Claim \ref{cla9_1}, $\partial(x_1,u)\leq 2g^*-3$. Hence $\partial(x,u)\leq 2g^*-2$ and $\partial(x,y)\leq 3g^*+2\leq g^*+12$.

If $x_3\in A_1$, then $x_2\in B_1$ and $y\notin I'_{42}\cup I'_{52}\cup I'_{62}\cup I'_8$, which implies $\partial(v,y)\leq g^*+2$. Since $x_2\in B_1$, we have $x_1\in J'_1\cup J'_2\cup J'_{31}\cup J'_{41}\cup J'_5$, and thus $x\rightarrow x_1$ and $\partial(x_1,u)\leq 7$, which implies $\partial(x,u)\leq 8$. Hence $\partial(x,y)\leq g^*+11$.

\vspace{0.1\baselineskip}
{\noindent \rm \bf Case 2.} $y\in I'$ and $x\in M'$.
\vspace{0.1\baselineskip}

If $y\notin I'_{83}$, then by Claims \ref{cla9.5}, \ref{cla14} and the proof of Claim \ref{cla14}, $\partial(v,y)\leq g^*+4$, and thus $\partial(x,y)\leq 2g^*+8$ by Table \ref{Table8-2}. 

Now we assume $y\in I'_{83}$.  Then $N(y)\cap ((\cup_{i=1}^4 A_i)\cup I)=\emptyset$ and $d(x,y)=4$. Let $P=xx_1x_2x_3y$ be a shortest $(x,y)$-path. Then $x_1\in X$, $x_2\in I_1\cup I_2$ and $x_3\in I'\cup A_5$. 

If $x_3\in I'$, then $x_3\in \cup_{i=1}^4 I'_i$, and thus $x_3\rightarrow y$ by Orientation \ref{cons12}. If $x_3\notin I'_{42}$, then $\partial(v,x_3)\leq g^*+2$ by Claims \ref{cla9.5} and \ref{cla14}, which implies $\partial(v,y)\leq g^*+3$ and $\partial(x,y)\leq 2g^*+7$. If $x_3\in I'_{42}$, then $x_2\rightarrow x_3\rightarrow y$ and $\partial(v,x_2)\leq 4$ by Orientation \ref{cons12}, Claims \ref{cla1} and \ref{cla7}, which implies $\partial(v,y)\leq 6$ and $\partial(x,y)\leq g^*+10$.

If $x_3\in A_5$, then $x_3\rightarrow y$ or $y\rightarrow x_3$ by (ii) of Orientation \ref{cons12}. If $x_3\rightarrow y$, then $\partial(v,y)\leq \partial(v,x_3)+\partial(x_3,y)\leq 5+1=6$ by Claim \ref{cla7}, which implies $\partial(x,y)\leq g^*+10$. If $y\rightarrow x_3$, then by Orientation \ref{cons12} and Definition \ref{def2}, there exists $y'\in I'_{83}$ such that $\cup_{i=5}^{9}A_i\rightarrow y'\rightarrow y\rightarrow\cup_{i=5}^{9}A_i$. Then $d(x,y')=4$, and let $Q=xy_1y_2y_3y'$ be a shortest $(x,y')$-path. Similarly, we have $y_1\in X$, $y_2\in I_1\cup I_2$ and $y_3\in I'\cup A_5$. If $y_3\in I'$, then $\partial(x,y')\leq 2g^*+7$ by the above arguments, and thus $\partial(x,y)\leq 2g^*+8\leq g^*+13$. If $y_3\in A_5$, then $y_3\to y'$, and similarly $\partial(x,y')\leq g^*+10$, which implies $\partial(x,y)\leq g^* +11$.

\vspace{0.1\baselineskip}
{\noindent \rm \bf Case 3.} $x\in J'$ and $y\in A$.
\vspace{0.1\baselineskip}

If $x\notin J'_{42}\cup J'_6$, then by Claim \ref{cla9_1}, we have $\partial(x,u)\leq 9$, and thus $\partial(x,y)\leq g^*+13$ by Table \ref{Table8-1}. If $y\notin A_7\cup A_8\cup A_9$, then by Proposition \ref{pro3}, we have $\partial(v,y)\leq 5$, and thus $\partial(x,y)\leq 17\leq g^*+13$ by Table \ref{Table8-1}. 

Now we assume $x\in J'_{42}\cup J'_6$ and $y\in A_7\cup A_8\cup A_9$. Then $N(x)\cap ((\cup_{i=1}^{6}B_i)\cup J)=\emptyset$ and $N(y)\cap (B\cup S_{2,2}\cup A_1\cup I\cup I')\subseteq (I_4\cup I_5\cup I_6\cup I')$. Let $y'\in N(y)\cap ((I_4\cup I_5\cup I_6)\cup I')$. Then $d(x,y')=4$, and let $P=xx_1x_2x_3y'$ be a shortest $(x,y')$ path. Thus we have $x_1\in J'$, $x_2\in B_1\cup J_1$ and $x_3\in A_1\cup I_1$. Since $x_2\in B_1\cup J_1$, we have $x_1\in J'_1\cup J'_2\cup J'_{31}\cup J'_{41}\cup J'_5$. Then $x\rightarrow x_1$ and $\partial(x_1,u)\leq 7$ by Orientation \ref{cons81} and Claim \ref{cla9_1}, and thus $\partial(x,u)\leq 8$, which implies $\partial(x,y)\leq g^*+12$.

\vspace{0.1\baselineskip}
{\noindent \rm \bf Case 4.} $x\in J'$ and $y\in I$.
\vspace{0.1\baselineskip}

If $y\notin I_5\cup I_{61}^{(22)}\cup I_{62}^{(2)}\cup I_{62}^{(3)}$, then by Claims \ref{cla1}, \ref{cla7} and \ref{cla9}, we have $\partial(v,y)\leq 5$, and thus $\partial(x,y)\leq 17$. If $x\notin J'_{32}\cup J'_{42}\cup J'_{6}$, then by Claim \ref{cla9_1}, we have $\partial(x,u)\leq 7$, and thus $\partial(x,y)\leq 17$.	

Now we assume $x\in J'_{32}\cup J'_{42}\cup J'_{6}$ and $y\in I_5\cup I_{61}^{(22)}\cup I_{62}^{(2)}\cup I_{62}^{(3)} $. Then $N(x)\cap ((\cup_{i=1}^6 B_i)\cup (\cup_{i=1}^3 J_i))=\emptyset$. Then $d(x,y)=4$ and let $P=xx_1x_2x_3y$ be a shortest $(x,y)$-path, and thus $x_1\in J'\cup (\cup_{i=4}^6 J_i)$, $x_2\in B_1$ and $x_3\in A_1$. Since $x_3\in A_1$, we have $y\in I_5\cup I_{61}^{(22)}$ by the definition of $I_{61}$, which implies $\partial(v,y)\leq g^*+2$ by Proposition \ref{pro1.5}. 

If $x\in J'_{32}\cup J'_{42}$, then $\partial(x,u)\leq 10$ by Claim \ref{cla9_1}, and thus $\partial(x,y)\leq g^*+13$. If $x\in J'_6$, then $x_1\in J'$ by the definition of $J'_6$. Since $x_2\in B_1$, we have $x_1\in J'_1\cup J'_2\cup J'_{31}\cup J'_{41}\cup J'_5$, which implies $\partial(x_1,u)\leq 7$ and $x\rightarrow x_1$ by Orientation \ref{cons81} and Claim \ref{cla9_1}. Hence $\partial(x,u)\leq 8$, and $\partial(x,y)\leq g^*+11$.

\vspace{0.1\baselineskip}
{\noindent \rm \bf Case 5.} $x\in J'$ and $y\in K'$.
\vspace{0.1\baselineskip}

If $x\notin J'_{32}\cup J'_{42}\cup J'_6$, then by Claim \ref{cla9_1}, we have $\partial(x,u)\leq 7$, and thus $\partial(x,y)\leq 17$.

Now we assume $x\in J'_{32}\cup J'_{42}\cup J'_6$. Then $N(x)\cap ((\cup_{i=1}^6 B_i)\cup (\cup_{i=1}^3 J_i))=\emptyset$, and thus $d(x,y)=4$. Let $P=xx_1x_2x_3y$ be a shortest $(x,y)$-path. Then $x_1\in J_4\cup J'$, $x_2\in L_1\cup J_1$ and $x_3\in K_1\cup I_1$. If $x_1\in J_4$, then $x\in J'_{32}$ and $x\rightarrow x_1$ by Orientation \ref{cons81}, which implies $\partial(x,u)\leq \partial(x,x_1)+\partial(x_1,u)\leq 1+5=6$ by Claim \ref{cla7}. Hence $\partial(x,y)\leq 16$. If $x_1\in J'$, then $x_2\in J_1$, and thus $x_1\in J'_1\cup J'_2$. By Orientation \ref{cons81} and Claim \ref{cla9_1}, $x\rightarrow x_1$ and $\partial(x_1,u)\leq 2g^*-3$, which implies $\partial(x,u)\leq 2g^*-2$. Therefore, $\partial(x,y)\leq 2g^*+8\leq g^*+13$.

\vspace{0.1\baselineskip}
{\noindent \rm \bf Case 6.} $x\in J'$ and $y\in S_{3,3}$.
\vspace{0.1\baselineskip}

If $x\notin J'_{63}$, then $\partial(x,u)\leq 10$ by Claim \ref{cla9_1} and its proof, which implies $\partial(x,y)\leq 17$. If $y\notin X'_2\cup X'_3$, then $\partial(v,y)\leq g^*+1$ by Claims \ref{cla4} and \ref{cla5}, which implies $\partial(x,y)\leq g^*+13$.

Now we assume $x\in J'_{63}$ and $y\in X'_2\cup X'_3$. Then $N(x)\cap ((\cup_{i=1}^6 B_i)\cup J)=\emptyset$, and thus $d(x,y)=4$. Let $P=xx_1x_2x_3y$ be a shortest $(x,y)$-path. Thus $x_1\in (\cup_{i=7}^{10} B_i)\cup J'$, $x_2\in B_2\cup J_1\cup J_2$ and $x_3\in S_{2,2}\cup X$. If $x_1\in J'$, then $x_1\in J'_1\cup J'_2\cup J'_{31}\cup J'_{41}\cup J'_5$ since $x_2\in B_2\cup J_1\cup J_2$, and thus $\partial(x_1,u)\leq 7$ and $x\rightarrow x_1$ by Claim \ref{cla9_1} and Orientation \ref{cons81}, which implies $\partial(x,u)\leq 8$. Hence $\partial(x,y)\leq 15$.

If $x_1\in \cup_{i=7}^{10} B_i$, then $x_2\in B_2$. For $x\rightarrow x_1$ and $x_1\in \cup_{i=7}^{9} B_i$, we have $x\rightarrow x_1\rightarrow x_2$ and $\partial(x_2,u)\leq 3$ by Orientation \ref{cons1} and Claim \ref{cla1}, and thus $\partial(x,u)\leq 5$, which implies $\partial(x,y)\leq 12$. For $x\rightarrow x_1$ and $x_1\in B_{10}$, we have $\partial(x,u)\leq \partial(x,x_1)+\partial(x_1,u)\leq 2g^*-1$ by Claim \ref{cla7_1}, which implies $\partial(x,y)\leq 2g^*+6$.  For $x_1\rightarrow x$, by Orientation \ref{cons81} and Definition \ref{def2}, there exists $x'\in J'_{63}$ such that $\cup_{i=7}^{10} B_i\rightarrow x\rightarrow x'\rightarrow\cup_{i=7}^{10} B_i$. Then $d(x',y)=4$, and let $Q=x'y_1y_2y_3y$ be a shortest $(x',y)$-path. Similarly, we have $y_1\in (\cup_{i=7}^{10} B_i)\cup J'$, $y_2\in B_2\cup J_1\cup J_2$ and $y_3\in S_{2,2}\cup X$. If $y_1\in J'$, then $\partial(x',y)\leq 15$ by the above arguments, and thus $\partial(x,y)\leq 16$. If $y_1\in \cup_{i=7}^{10} B_i$, then $x'\rightarrow y_1$, and similarly $\partial(x',y)\leq 2g^*+6$, which implies $\partial(x,y)\leq 2g^*+7\leq g^*+12$.

\vspace{0.1\baselineskip}
{\noindent \rm \bf Case 7.} $x\in J'$ and $y\in M$.
\vspace{0.1\baselineskip}

If $x\notin J'_{42}\cup J'_6$, then by Claim \ref{cla9_1}, $\partial(x,u)\leq 9$, and thus $\partial(x,y)\leq 17$. If $y\in M_1\cup M_2\cup M_{31}$, then $\partial(v,y)\leq 5$ by Claim \ref{cla6}, which implies $\partial(x,y)\leq 17$.

Now we assume $x\in  J'_{42}\cup J'_6$ and $y\in M_{32}$. Then $N(x)\cap (J\cup (\cup_{i=1}^6 B_i))=\emptyset$ and $N(y)\cap (K\cup L)=\emptyset$, and thus $d(x,y)=4$. Let $P=xx_1x_2x_3y$ be a shortest $(x,y)$-path. Then $x_1\in J'$, $x_2\in J_1\cup J_2$ and $x_3\in X$, and thus $x_1\in J'_1\cup J'_2$. Then $x\rightarrow x_1$ and $\partial(x_1,u)\leq 2g^*-3$ by Orientation \ref{cons81} and Claim \ref{cla9_1}, and thus $\partial(x,u)\leq 2g^*-2$. Hence $\partial(x,y)\leq 2g^*+6$ by Table \ref{Table8-2}.

\vspace{0.1\baselineskip}
{\noindent \rm \bf Case 8.} $x\in J'$ and $y\in M'$.
\vspace{0.1\baselineskip}

If $x\notin J'_{32}\cup J'_{42}\cup J'_6$, then by Claim \ref{cla9_1}, $\partial(x,u)\leq 7$, and thus $\partial(x,y)\leq 16$. If $y\in M'_1$, then by Claim \ref{cla6}, $\partial(v,y)\leq g^*$, and thus $\partial(x,y)\leq g^*+12$.

Now we assume $x\in J'_{32}\cup  J'_{42}\cup J'_6$ and $y\in M'_2$. Then $N(x)\cap ((\cup_{i=1}^3 J_i)\cup (\cup_{i=1}^6 B_i))=\emptyset$, and thus $d(x,y)=4$. Let $P=xx_1x_2x_3y$ be a shortest $(x,y)$-path. Then $x_1\in (J_4\cup J_5)\cup J'$, $x_2\in L\cup J_1\cup J_2$ and $x_3\in M\cup X$. 

If $x_1\in J'$, then $x_1\in J'_1\cup J'_2$, and thus $x\rightarrow x_1$ and $\partial(x_1,u)\leq 2g^*-3$ by Orientation \ref{cons81} and Claim \ref{cla9_1}. Hence $\partial(x,u)\leq 2g^*-2$, and $\partial(x,y)\leq 2g^*+7$.

If $x_1\in J_4\cup J_5$, then $x\in J'_{32}$ and $x_2\in L\cup J_2$. By Orientations \ref{cons1} and \ref{cons81}, $x\rightarrow x_1\rightarrow x_2\rightarrow x_3$, and thus $\partial(x,x_3)\leq 3$. For $x_3\in X$, we have $\partial(x,u)\leq \partial(x,x_3)+\partial(x_3,u)\leq 3+5=8$ by Claim \ref{cla4}, which implies $\partial(x,y)\leq 17$. For $x_3\in M$, we have $y\in M'_{21}$, and $x_3\rightarrow y$ by Orientation \ref{cons6}, and thus $\partial(x,y)\leq \partial(x,x_3)+\partial(x_3,y)\leq 4$.
{\hfill $\blacksquare$\par}
	
	\section*{Funding}
	This work is supported by the National Natural Science Foundation of China (Grant Nos. 12371347, 12271337, 12401447).

	\vspace{0.5em}
	
	\bibliographystyle{unsrt}

\end{document}